\documentclass[reqno]{amsart}
\usepackage{amsmath, amsthm, amscd, amssymb, amsfonts, amsbsy}
\usepackage{latexsym, color, enumerate}
\usepackage{pxfonts}
\usepackage[dvipsnames]{xcolor}
\usepackage{bbm}

\theoremstyle{plain}
\newtheorem{theorem}[equation]{Theorem}
\newtheorem{lemma}[equation]{Lemma}

\theoremstyle{definition}

\theoremstyle{remark}
\newtheorem{remark}[equation]{Remark}

\newcommand{\dv}{\operatorname{div}}

\newcommand{\dist}{\operatorname{dist}}
\newcommand{\diam}{\operatorname{diam}}

\newcommand{\tr}{\operatorname{tr}}

\numberwithin{equation}{section}

\newcommand{\bR}{\mathbb{R}}

\newcommand\cD{\mathcal{D}}

\providecommand{\set}[1]{\{#1\}}
\providecommand{\Set}[1]{\left\{#1\right\}}

\providecommand{\abs}[1]{\lvert#1\rvert}
\providecommand{\Abs}[1]{\left\lvert#1\right\rvert}

\providecommand{\norm}[1]{\lVert#1\rVert}

\renewcommand{\vec}[1]{\boldsymbol{#1}}

\begin{document}
\title[Regularity of elliptic equations in double divergence form]
{Regularity of elliptic equations in double divergence form and applications to Green's function estimates}

\author[J. Choi]{Jongkeun Choi}
\address[J. Choi]{Department of Mathematics Education, Pusan National University,  Busan, 46241, Republic of Korea}
\email{jongkeun\_choi@pusan.ac.kr}
\thanks{J. Choi was supported by the National Research Foundation of Korea (NRF) under agreement NRF-2022R1F1A1074461.}

\author[H. Dong]{Hongjie Dong}
\address[H. Dong]{Division of Applied Mathematics, Brown University,
182 George Street, Providence, RI 02912, United States of America}
\email{Hongjie\_Dong@brown.edu}
\thanks{H. Dong was partially supported by the NSF under agreement DMS-2055244.}

\author[D. Kim]{Dong-ha Kim}
\address[D. Kim]{Department of Mathematics, Yonsei University, 50 Yonsei-ro, Seodaemun-gu, Seoul 03722, Republic of Korea}
\email{skyblue898@yonsei.ac.kr}

\author[S. Kim]{Seick Kim}
\address[S. Kim]{Department of Mathematics, Yonsei University, 50 Yonsei-ro, Seodaemun-gu, Seoul 03722, Republic of Korea}
\email{kimseick@yonsei.ac.kr}
\thanks{S. Kim is supported by the National Research Foundation of Korea (NRF) under agreement NRF-2022R1A2C1003322.}



\begin{abstract}
We investigate the regularity of elliptic equations in double divergence form, where the leading coefficients satisfying the Dini mean oscillation condition. We prove that the solutions are differentiable on the zero level set and derive a pointwise bound for the derivative, which substantially improve a recent result by Leit\~ao, Pimentel,  and Santos (Anal. PDE 13(4):1129--1144, 2020).
As an application, we establish global pointwise estimates for the Green's function of second-order uniformly elliptic operators in non-divergence form, considering Dini mean oscillation coefficients in bounded $C^{1,\alpha}$ domains. This result extends a recent work by Chen and Wang (Electron. J. Probab. 28(36):54 pp,  2023).
\end{abstract}
\maketitle

\section{Introduction and main results}
We consider the elliptic operator $L^*$ of the form
\[
L^* u= \sum_{i,j=1}^d D_{ij}(a^{ij}u) - \sum_{i=1}^d D_i(b^i u) + cu=\dv^2 (\mathbf A u)-\dv(\vec bu) + cu,
\]
defined on a domain $\Omega \subset \bR^d$ with $d \ge 2$.
We assume that the coefficients $\mathbf{A}:=(a^{ij})$ are symmetric and satisfy the uniform ellipticity condition:
\begin{equation}					\label{ellipticity-nd}
a^{ij}=a^{ji},\quad \lambda \abs{\xi}^2 \le \sum_{i,j=1}^d a^{ij}(x) \xi^i \xi^j \le \Lambda \abs{\xi}^2,\quad \forall x \in \Omega,
\end{equation}
for some positive constant $\lambda$ and $\Lambda$.
We will impose additional assumptions on the coefficients $\mathbf A$, $\vec b$, and $c$ shortly.

The operator $L^*$ is often referred to as a double divergence form operator and is the formal adjoint of the elliptic operator in non-divergence form $L$, given by
\[
Lv=\sum_{i,j=1}^d a^{ij} D_{ij}v + \sum_{i=1}^d b^i D_i  v + cv=\tr(\mathbf A D^2 v)+ \vec b \cdot Dv +cv.
\]
An important example of a double divergence form equation is the stationary Kolmogorov equation for invariant measures of a diffusion process; see \cite{BKRS15}.

In this article, we are concerned with the regularity of weak solutions of $L^*u=0$.
By a weak solution, we mean a distributional solution that is locally integrable.
In the context of regularity, there is a notable distinction between equations in double divergence form and those in divergence or non-divergence form.
As noted in \cite{BS17}, solutions of $\dv^2(\mathbf A u)=0$ are not more regular than $\mathbf A$ itself.
For instance, consider the one-dimensional equation $(au)'' = 0$.
For any affine function $\ell$, the function $u = \ell /a$ is a solution.
Therefore, if the coefficient $a$  is merely measurable, then so is the function $u$.

This raises a natural question: does the continuity of the coefficient $\mathbf A$ imply the continuity of a solution to $\dv^2(\mathbf A u)=0$?
The answer is known to be negative.
There exists a weak solution of $\dv^2 (\mathbf A u)=0$ that is unbounded even when $\mathbf A$ is uniformly continuous.
However, if $\mathbf A$ is Dini continuous, then a weak solution of $L^*u=0$ is continuous.
More precisely, a weak solution has a continuous representative.
In particular, if $\mathbf A$ is $C^\alpha$-continuous with $\alpha\in (0,1)$, then the solution is also H\"older continuous with the same exponent $\alpha$; see \cite{Sjogren73, Sjogren75}.

More recently, it has been shown that if $\mathbf A$ has Dini mean oscillation and $\vec b$ and $c$ belong to certain Lebesgue classes, then a weak solution of $L^*u=0$ is continuous; see \cite{DK17, DEK18}.
A Dini continuous function has Dini mean oscillation, but the converse is not necessarily true.

In an interesting recent paper \cite{LPS20}, the authors demonstrated that if $\mathbf A$ is H\"older continuous and $u$ is a weak solution of $\dv^2(\mathbf Au)=0$ with $u(x^o)=0$, then $u$ is $C^{1-\varepsilon}$ at $x^o$ for any $\varepsilon\in (0,1)$.
More precisely, there exists a constant $C>0$ such that
\[
\sup_{B_r(x^o)}\, \abs{u} = \sup_{B_r(x^o)}\, \abs{u-u(x^o)} \le C  r^{1-\varepsilon},\;\text{for every }\varepsilon \in (0,1).
\]
Recall that for double divergence elliptic equations, $u$ generally cannot have better regularity than $\mathbf A$.
Therefore, the result in \cite{LPS20} suggests that the solution exhibits improved regularity on its nodal set.

It is natural to ask if $u$ is differentiable at $x^o$ under the  assumptions above, and whether the H\"older continuity condition on $\mathbf A$ can be relaxed to the Dini continuity condition or even the Dini mean oscillation condition.

In this article, we provide affirmative answers to both questions, substantially improving the result in \cite{LPS20}.
We also present an analogous result when $x^o$ lies on the boundary, assuming the boundary  is $C^{1,\alpha}$ for some $\alpha \in (0,1)$.
Moreover, in the case where $\mathbf A$ is $C^\alpha$, as assumed in \cite{LPS20}, our proof reveals a stronger and optimal pointwise $C^{1,\alpha}$ regularity:
\[
\abs{u(x)-Du(x^o)\cdot(x-x^o)}=O(\abs{x-x^o}^{1+\alpha}) \quad\text{as}\quad x\to x^o.
\]

We  became interested in this problem after learning about the recent paper \cite{CW23} on the Green's function estimates for elliptic operators in non-divergence form with Dini continuous coefficients.
Let $L_0$ be the operator given by $L_0 u=\tr(\mathbf A D^2u)$, where $\mathbf A$ is a Dini continuous coefficient.
Then, the Green's function $G(x,y)$ for the operator $L_0$ in a $C^{1,1}$ domain $\Omega \subset \bR^d$ ($d\ge 3$) satisfies the upper bound
\[
G(x,y) \le \frac{C}{\abs{x-y}^{d-2}} \left(1 \wedge \frac{d_x}{\abs{x-y}} \right) \left(1 \wedge \frac{d_y}{\abs{x-y}} \right),
\]
where $a\wedge b=\min(a,b)$ and $d_x=\dist(x,\partial\Omega)$.
This implies, in particular, that  $G(x.\,\cdot\,)$ has Lipschitz decay near the boundary.
Since $L^*G(x,\,\cdot\,)=0$ away from $x$ and vanishes on $\partial\Omega$, our results suggest that $G(x.\,\cdot\,)$ should exhibit Lipschitz decay near the boundary when $\mathbf A$ has Dini mean oscillation.

As an application of our result, we  establish the above bound for the Green's function under the weaker assumptions that $\mathbf A$ has Dini mean oscillation and  $\Omega$ is a $C^{1,\alpha}$ domain for some $\alpha \in (0,1)$.

Prior to presenting our main results, we will outline the conditions imposed on the coefficients.
We say that a function $f$ defined on $\Omega$ has Dini mean oscillation and write $f \in \mathrm{DMO}(\Omega)$ if the mean oscillation function $\omega_f: \bR_+ \to \bR$, defined by
\[
\omega_f(r):=\sup_{x\in \Omega} \fint_{\Omega \cap B_r(x)} \,\abs{f(y)-(f)_{\Omega \cap B_r(x)}}\,dy, \;\; \text{where }(f)_{\Omega \cap B_r(x)}=\fint_{\Omega \cap B_r(x)} f,
\]
satisfies the Dini condition, i.e.,
\[
\int_0^1 \frac{\omega_f(t)}t \,dt <+\infty.
\]
It is evident that if $f$ is Dini continuous, then $f$ has Dini mean oscillation.
However, it is noteworthy that the Dini mean oscillation condition is  less restrictive than Dini continuity; see \cite[p. 418]{DK17} for a concrete example.
Moreover, if $f$ has Dini mean oscillation, it follows that $f$ is uniformly continuous, with its modulus of continuity governed by $\omega_f$, as detailed in the Appendix of \cite{HK20}.

We would also like to point out that there is related literature addressing continuity conditions weaker than Dini continuity, which differ from the Dini mean oscillation condition.
For example, see \cite{MMcO2010, MMcO2022}.

In what follows, we consider the more general inhomogeneous equation:
\[
L^* u=\dv^2 \mathbf{f}+\dv \vec g+ h\quad\text{in }\;\Omega,
\]
where $\mathbf f =(f^{ij})$ is a symmetric $d \times d$ matrix valued function, $\vec g=(g^1,\ldots, g^d)$ is a vector valued function, and $h$ is a scalar function.

Hereafter, we impose the following assumptions:
\begin{equation}			\label{cond_lower}
\mathbf{A} \in \mathrm{DMO}(\Omega), \;\; \vec b \in L^{p_0}(\Omega), \;\text{ and }\;c \in L^{p_0/2}(\Omega)\;\text{ for some }p_0>d.
\end{equation}
Additionally, we assume:
\begin{equation}			\label{cond_dat}
D \mathbf f \in \mathrm{DMO}(\Omega),\;\; \vec g \in \mathrm{DMO}(\Omega), \;\text{ and }\; h\in L^{p_0}(\Omega)\,\text{ for some }p_0>d.
\end{equation}
These conditions ensure that the following functions satisfy the Dini condition (see Remark~\ref{rmk_omega}):
\begin{align}
					\label{omega_coef}
\omega_{\rm coef}(r)&:=  \omega_{\mathbf A}(r)+  r\sup_{x \in \Omega} \fint_{\Omega \cap B_r(x)}\abs{\vec b}+ r^2\sup_{x \in \Omega} \fint_{\Omega \cap B_r(x)}\abs{c},\\
					\label{omega_dat}
\omega_{\rm dat}(r)&:=  \omega_{D \mathbf f}(r)+\omega_{\vec g}(r) + r \sup_{x \in \Omega} \fint_{B_r(x) \cap \Omega}\abs{h}.
\end{align}

We say that  $u\in L^1_{\rm loc}(\Omega)$ is a weak solution (or an adjoint solution) of
\[
L^* u=\dv^2 \mathbf{f}+\dv \vec g+ h\quad\text{in }\;\Omega,
\]
if $\vec b u $, $cu \in L^1_{\rm loc}(\Omega)$, and for all $\varphi \in C^\infty_c(\Omega)$, the following holds:
\[
\int_\Omega u a^{ij} D_{ij} \varphi + u b^i D_i \varphi +u c\varphi=\int_\Omega f^{ij} D_{ij}\varphi-g^i  D_i\varphi+h\varphi,
\]
where the usual summation convention over repeated indices is applied.

Now, we state our main results.
We use the following notation:
\[
d_x=\dist(x, \partial \Omega),\quad a\wedge b= \min(a,b),
\]
and $C=C(\alpha, \beta, \cdots)$ to indicate that $C$ is a quantity depending on $\alpha$, $\beta$, $\cdots$.
Additionally, we write $A \lesssim B$ if $A \le c B$ for some $c>0$.

\begin{theorem}			\label{thm-main01}
Let $\Omega \subset \bR^d$ be a domain, and assume the conditions \eqref{ellipticity-nd}--\eqref{cond_dat} hold.
Let $u \in L^\infty_{\rm loc}(\Omega)$ be a weak solution of
\[
L^* u=\dv^2 \mathbf{f}+\dv \vec g+ h\quad\text{in }\;\Omega.
\]
If $u(x^o)=0$ for some $x^o\in\Omega$, then $u$ is differentiable at $x^o$.
More precisely, there exist moduli of continuity $\varrho_{\rm coef}$ and $\varrho_{\rm dat}$ determined by $\omega_{\rm coef}$ and $\omega_{\rm dat}$ in \eqref{omega_coef} and \eqref{omega_dat}, respectively, such that for $R_0:=r_0 \wedge \frac12 d_{x^o}$ and $0<\abs{x-x^o}<\frac12 R_0$, the following estimate holds:
\[
\frac{\abs{u-Du(x^o)\cdot (x-x^o)}}{\abs{x-x^o}} \le C  \varrho_{\rm coef}(\abs{x-x^o})\left(\frac{1}{R_0} \fint_{B_{2R_0}(x^o)} \abs{u} + \varrho_{\rm dat}(R_0) \right)+ C \varrho_{\rm dat}(\abs{x-x^o}),
\]
where $C$ and $r_0$ are positive constants depending only on $d$, $\lambda$, $\Lambda$, and $\omega_{\rm coef}$.
In particular, if $\omega_{\rm coef}(r) \lesssim r^\alpha$ for some $\alpha \in (0,1)$, then $\varrho_{\rm coef}(r) \lesssim r^\alpha$, and $\varrho_{\rm dat}\equiv 0$ if the inhomogeneous terms $\mathbf f$, $\vec g$, and $h$ are all zero.
Moreover, the derivative of $u$ at $x^o$ satisfies the estimate:
\[
\abs{D u (x^o)} \le C \left( \frac{1}{R_0} \fint_{B_{2R_0}(x^o)} \abs{u} + \varrho_{\rm dat}(R_0)\right).
\]
\end{theorem}

In the next theorem, we assume that $\Omega$ is a bounded $C^{1,\alpha}$ domain for some $\alpha \in (0,1)$, with unit exterior normal vector $\nu$.
We say that $u \in L^1(\Omega)$ is a weak solution (or an adjoint solution) of
\[
L^* u=\dv^2 \mathbf{f}+\dv \vec g+ h\quad\text{in }\;\Omega,\qquad
u=\frac{\mathbf f  \nu \cdot  \nu}{\mathbf A  \nu \cdot  \nu}\quad\text{on }\;\partial\Omega,
\]
if $\vec b u$, $cu \in L^1(\Omega)$, and for all $v\in C^2(\overline \Omega)$ satisfying $v=0$ on $\partial\Omega$, we have
\[
\int_\Omega u Lv=\int_\Omega \tr(\vec f D^2v)-\vec g \cdot Dv+hv.
\]

\begin{theorem}			\label{thm-main02}
Let $\Omega \subset \bR^d$ be a bounded $C^{1,\alpha}$ domain for some $\alpha \in (0,1)$, and assume that the conditions \eqref{ellipticity-nd} --  \eqref{cond_dat} are satisfied.
Let $u \in L^\infty(\Omega)$ be a weak solution of
\[
L^* u=\dv^2 \mathbf{f}+\dv \vec g+ h\quad\text{in }\;\Omega,\qquad
u=\frac{\mathbf f \nu \cdot \nu}{\mathbf A \nu \cdot  \nu}\quad\text{on }\;\partial\Omega.
\]
Suppose that $u(x^o)=0$ for some $x^o \in \partial\Omega$, and that the compatibility condition $\mathbf{f}(x^o)=0$ is satisfied.
Then, $u$ is differentiable at $x^o$.
Furthermore, there exist moduli of continuity $\varrho_{\rm coef}$ and $\varrho_{\rm dat}$, determined by $\omega_{\rm coef}$ and $\omega_{\rm dat}$ in  \eqref{omega_coef} and \eqref{omega_dat}, respectively, as well as by the $C^{1,\alpha}$ properties of $\Omega$, such that for $0<\abs{x-x^o}<\frac14 r_0$, the following estimate holds: 
\[
\frac{\abs{u-Du(x^o)\cdot (x-x^o)}}{\abs{x-x^o}} \le C  \varrho_{\rm coef}(\abs{x-x^o})\left(\frac{1}{r_0} \fint_{B_{4r_0}(x^o) \cap \Omega} \abs{u} + \varrho_{\rm dat}(r_0) \right)+ C \varrho_{\rm dat}(\abs{x-x^o}),
\]
where $C$ and $r_0$ are positive constants depending only on $d$, $\lambda$, $\Lambda$, $\omega_{\rm coef}$, and the $C^{1,\alpha}$ properties of $\Omega$.
Moreover, the derivative of $u$ at $x^o$ satisfies the estimate:
\[
\abs{D u (x^o)} \le C \left( \frac{1}{r_0} \fint_{B_{4r_0}(x^o)\cap \Omega} \abs{u} + \varrho_{\rm dat}(r_0)\right).
\]
\end{theorem}

\begin{remark}			\label{rmk_omega}
Let $\omega_{\rm coef}(r)$ and $\omega_{\rm dat}(r)$ be as defined in \eqref{omega_coef} and \eqref{omega_dat}, respectively.
If $\vec b \in L^{p_0}(\Omega)$, $c\in L^{p_0/2}(\Omega)$, and $h\in L^{p_0}(\Omega)$ for some $p_0>d$, then by H\"older's inequality, we obtain the following estimates:
\begin{align*}
\omega_{\rm coef}(r) &\lesssim  \omega_{\mathbf A}(r) +  r^{1-\frac{d}{p_0}} \norm{\vec b}_{L^{p_0}(\Omega)}+  r^{2-\frac{2d}{p_0}} \norm{c}_{L^{p_0/2}(\Omega)},\\
\omega_{\rm dat}(r) & \lesssim \omega_{D \mathbf f}(r)+\omega_{\vec g}(r) +  r^{1-\frac{d}{p_0}} \norm{h}_{L^{p_0}(\Omega)},
\end{align*}
which implies that both $\omega_{\rm coef}(r)$ and $\omega_{\rm dat}(r)$ satisfy the Dini condition.
The proof of Theorem~\ref{thm-main01} will show that, in order to estimate $Du(x^o)$, it is sufficient to assume that $\mathbf{A}$, $D \mathbf{f}$, and $\vec{g} \in \mathrm{DMO}(\Omega)$, and that $\vec{b}$, $c$, and $h \in L^p_{\rm loc}(\Omega)$ for some $p > 1$, along with the Dini conditions on $\omega_{\rm coef}(r)$ and $\omega_{\rm dat}(r)$, instead of requiring the conditions in \eqref{cond_lower} and \eqref{cond_dat}. Similarly, the proof of Theorem~\ref{thm-main02} shows that it is only necessary to assume these conditions to conclude the result of the theorem.
\end{remark}

The paper is organized as follows: In Section 2, we provide the proof of Theorem \ref{thm-main01} under the assumption that the lower-order coefficients and inhomogeneous terms are zero. The proof of Theorem \ref{thm-main01} for the general case is presented in Section \ref{sec3}. We adopt this two-step approach to convey the main idea of the proof without delving into excessive technicalities. In Section \ref{sec4}, we present a proof for Theorem \ref{thm-main02}. Section \ref{sec4} is divided into two subsections: in Section \ref{sec:flat}, we assume that the boundary portion containing $x^o$ is flat, and in Section \ref{sec:c1alpha}, we address the general case. This structure is designed to present the main ideas first and defer the technicalities to later. The application to estimates for Green's function is presented as Theorem \ref{230321_thm1} in Section \ref{sec5}.

\section{Proof of Theorem~\ref{thm-main01}: A simple case}		\label{sec2}
In this section, we deal with an elliptic operator $L_0^*$ in the form
\[
L_0^* u=  D_{ij}(a^{ij} u)=\dv^2 (\mathbf Au).
\]
Here and throughout, we use the usual summation convention over repeated indices.
We also assume that the inhomogeneous terms $\mathbf f$, $\vec g$, and $h$ are all zero.
These assumptions allow us to present the main idea of the proof more clearly.

By \cite[Theorem~1.10]{DK17}, any locally bounded weak solution of $L_0^* u =0$ in $\Omega$ is continuous in $\Omega$.
We shall show that if $u(x^o)=0$ for some $x^o \in \Omega$, then $u$ is differentiable at $x^o$.
For simplicity, let us assume that $x^o=0$ and $d_{x^o} \ge 1$, so that $B_{1}(0) \subset \Omega$.

\begin{remark}			\label{rmk_rescaling}
This can be achieved through an affine change of coordinates.
Note that the new coefficients $\tilde{\mathbf A}$ still satisfy the Dini mean oscillation condition as well as the uniform ellipticity condition \eqref{ellipticity-nd}.
Moreover, $\omega_{\tilde{\mathbf A}}$ only improves under dilation.
Therefore, the constants that appear below in the proof do not depend on $d_{x^0}$.
\end{remark}

Let us denote by $\bar{\mathbf A}=(\mathbf A)_{B_r}$, the average of $\mathbf A$ over the ball $B_r=B_r(0)$, where $r \in (0,\frac12]$.
Decompose $u$ as $u=v+w$, where $w \in L^p(B_r)$ (for some $p>1$) is the weak solution of the problem
\[
\left\{
\begin{aligned}
\dv^2(\bar{\mathbf A}w) &= -\dv^2((\mathbf{A}-\bar{\mathbf A})u)\;\mbox{ in }\; B_r,\\
(\bar{\mathbf A} \nu \cdot \nu) w&= -(\mathbf{A}-\bar{\mathbf A})u \nu\cdot \nu  \;\mbox{ on }\;\partial B_r.
\end{aligned}
\right.
\]
It follows from \cite[Lemma 2.23]{DK17} that $w \in L^{\frac12}(B_r)$ and
\begin{equation}			\label{eq1622thu}
\left(\fint_{B_r} \abs{w}^{\frac12}\right)^{2} \le C \omega_{\mathbf A}(r) \norm{u}_{L^\infty(B_r)},
\end{equation}
where $C=C(d,\lambda, \Lambda)$.
In fact, we have $w \in L^p(B_r)$ for any $p \in (0,1)$, but we fix $p=\frac12$ for the sake of definiteness.

Note that $v=u-w$ satisfies
\[
\dv^2(\bar{\mathbf A} v)=\tr(\bar{\mathbf A} D^2 v)=0\;\text { in }\;B_r.
\]
By the interior regularity estimates for solutions of elliptic equations with constant coefficients, we have $v \in C^{\infty} (B_r)$.
In particular,
\[
[v]_{C^{1,1}(B_{r/2})} \le \frac{C}{r^2} \left(\fint_{B_r} \abs{v}^{\frac12}\right)^{2},
\]
where $C=C(d,\lambda, \Lambda)$.
Moreover, the above estimate remains valid if $v$ is replaced by $v-\ell$, where $\ell(x)$ is any affine function.
That is,
\[
[v]_{C^{1,1}(B_{r/2})} \le \frac{C}{r^2} \left(\fint_{B_r} \abs{v-\ell}^{\frac12}\right)^{2},
\]
for any $\ell(x)=\vec a \cdot x + b$ with $\vec a \in \bR^d$ and $b \in \bR$.

On the other hand, by virtue of Taylor's formula, for any $\rho \in (0, r]$, we have
\[
\sup_{x\in B_\rho}\;\abs{v(x)-Dv(0)\cdot x - v(0)}  \le C(d) [v]_{C^{1,1}(B_\rho)} \,\rho^2.
\]

Therefore, for any $\kappa \in (0, \frac12)$ and an affine function $\ell(x)=\vec a \cdot x +b$, we obtain
\begin{equation}			\label{eq1616fri}
\left(\fint_{B_{\kappa r}} \abs{v(x)-Dv(0)\cdot x -v(0)}^{\frac12}dx\right)^{2} \le C [v]_{C^{1,1}(B_{r/2})} (\kappa r)^2 \le C \kappa^2 \left(\fint_{B_r} \abs{v-\ell}^{\frac12}\right)^{2},
\end{equation}
where $C=C(d,\lambda, \Lambda)>0$.

Let us define
\begin{equation}			\label{eq1100sat}
\varphi(r):=\frac{1}{r} \,\inf_{\substack{\vec a \in \bR^d\\b \in \bR}} \left(\fint_{B_{r}} \abs{u(x)-\vec a \cdot x -b}^{\frac12}\,dx\right)^{2}.
\end{equation}
Since $u=v+w$, we use the quasi-triangle inequality, \eqref{eq1622thu}, and \eqref{eq1616fri}  to obtain
\begin{align}			\nonumber
\kappa r \varphi(\kappa r) & \le \left(\fint_{B_{\kappa r}} \abs{u(x)-Dv(0)\cdot x -v(0)}^{\frac12}dx\right)^{2}\\					\nonumber
&\le C\left(\fint_{B_{\kappa r}} \abs{v(x)-Dv(0)\cdot x -v(0)}^{\frac12}dx\right)^{2} + C \left(\fint_{B_{\kappa r}} \abs{w(x)}^{\frac12}dx\right)^{2}\\
					\nonumber
& \le  C \kappa^2 \left(\fint_{B_r} \abs{v-\ell}^{\frac12}\right)^{2} + C \kappa^{-2d}\left(\fint_{B_{r}} \abs{w}^{\frac12}\right)^{2}\\	
					\nonumber
& \le C \kappa^2 \left(\fint_{B_r} \abs{u-\ell}^{\frac12}\right)^{2} + C \left(\kappa^2+\kappa^{-2d}\right)\left(\fint_{B_{r}} \abs{w}^{\frac12}\right)^{2}\\
					\label{eq0224sat}
& \le C \kappa^2 \left(\fint_{B_r} \abs{u-\ell}^{\frac12}\right)^{2} + C \left(\kappa^2+\kappa^{-2d}\right) \omega_{\mathbf A}(r) \norm{u}_{L^\infty(B_r)}.
\end{align}

Let $\beta \in (0,1)$ be an arbitrary but fixed number.
With this $\beta$, choose $\kappa=\kappa(d, \lambda,\Lambda, \beta) \in (0, \frac12)$ such that $C \kappa \le  \kappa^{\beta}$.
Then, we have
\begin{equation}			\label{eq1110sat}
\varphi(\kappa r) \le \kappa^\beta \varphi(r) + C \omega_{\mathbf A}(r)\,\frac{1}{r} \norm{u}_{L^\infty(B_r)},
\end{equation}
where $C=C(d, \lambda, \Lambda, \kappa)=C(d, \lambda, \Lambda, \beta)$.

Let $r_0 \in (0, \frac12]$ be a number to be chosen later.
By iterating \eqref{eq1110sat}, we have, for $j=1,2,\ldots$,
\[
\varphi(\kappa^j r_0) \le \kappa^{\beta j} \varphi(r_0) + C \sum_{i=1}^{j} \kappa^{(i-1)\beta} \omega_{\mathbf A}(\kappa^{j-i} r_0)\,\frac{1}{\kappa^{j-i} r_0}\, \norm{u}_{L^\infty(B_{\kappa^{j-i} r_0})}.
\]

By defining
\begin{equation}			\label{eq1244sat}
M_j(r_0):=\max_{0\le i < j}\, \frac{1}{\kappa^i r_0} \,\norm{u}_{L^\infty(B_{\kappa^i r_0})} \quad \text{for }j=1,2, \ldots,
\end{equation}
we obtain
\begin{equation}			\label{eq1114sat}
\varphi(\kappa^j r_0) \le \kappa^{\beta j} \varphi(r_0) + C M_j(r_0) \tilde \omega_{\mathbf A}(\kappa^j r_0),
\end{equation}
where, similar to \cite[(2.15)]{DK17}, we define
\begin{equation}			\label{eq1245sat}
\tilde\omega_{\mathbf A}(t):= \sum_{i=1}^\infty \kappa^{i\beta} \left\{ \omega_{\mathbf A}(\kappa^{-i} t) [ \kappa^{-i} t \le 1] +\omega_{\mathbf A}(1)[\kappa^{-i} t >1]\right\}.
\end{equation}
Here, we use Iverson bracket notation:  $[P] = 1$ if $P$ is true, and $[P] = 0$ otherwise.

We emphasize that $\tilde\omega_{\mathbf A}(t)$ satisfies the Dini condition whenever $\omega_{\mathbf A}(t)$ does.
Specifically, it is evident  $\tilde \omega_{\mathbf A}(t) \lesssim t^\alpha$ if $\mathbf A \in C^\alpha$ for some $\alpha \in (0,1)$.

\begin{remark}				\label{rmk1147}
We observe that $\omega_{\mathbf A}(t) \lesssim \tilde \omega_{\mathbf A}(t)$.
Additionally, we will use the following fact:
\[
\sum_{j=0}^\infty \omega_{\mathbf A}(\kappa^j r) \lesssim \int_0^{r} \frac{\omega_{\mathbf A}(t)}{t}\,dt,
\]
which, in particular, implies that $\omega_{\mathbf A}(r) \lesssim \int_0^r \omega_{\mathbf A}(t)/t\,dt$; see \cite[Lemma 2.7]{DK17}.
\end{remark}

For each fixed $r$, the infimum in \eqref{eq1100sat} is realized by some $\vec a \in \bR^d$ and $b \in \bR$.
For $j=0,1,2,\ldots$, let $\vec a_j \in \bR^d$ and $b_j \in \bR$ be chosen such that
\begin{equation}			\label{eq0203tue}
\varphi(\kappa^j r_0)= \frac{1}{\kappa^j r_0}
\left(\fint_{B_{\kappa^j r_0}} \abs{u(x)-\vec a_j \cdot x -b_j}^{\frac12}dx\right)^{2}.
\end{equation}

From \eqref{eq1100sat} and H\"older's inequality, we have
\begin{equation}			\label{eq0538tue}
\varphi(r_0) \le \frac{1}{r_0} \fint_{B_{r_0}} \abs{u}.
\end{equation}
Combining \eqref{eq1114sat} and \eqref{eq0538tue}, we deduce
\begin{equation}\label{eq4.38}
\varphi(\kappa^j r_0) \le  \frac{\kappa^{\beta j}}{r_0} \fint_{B_{r_0}} \abs{u}+ CM_j(r_0) \tilde \omega_{\mathbf A}(\kappa^j r_0).
\end{equation}

Next, observe that for $j=0,1,2,\ldots$, we have
\begin{equation}\label{eq0218tue}
\fint_{B_{\kappa^j r_0}} \abs{\vec a_j \cdot x +b_j}^{\frac12}dx
\le \fint_{B_{\kappa^j r_0}} \abs{u-\vec a_j \cdot x -b_j}^{\frac12}dx+
\fint_{B_{\kappa^j r_0}} \abs{u}^{\frac12}\le 2\fint_{B_{\kappa^j r_0}} \abs{u}^{\frac12}.
\end{equation}
Furthermore,
\begin{align*}
\abs{b_j}^{\frac12} = \abs{\vec a_j \cdot x + b_j - 2(\vec a_j\cdot x/2 + b_j)}^{\frac12} \le \abs{\vec a_j \cdot x + b_j}^{\frac12} + 2^{\frac12}\, \abs{\vec a_j \cdot x/2 + b_j}^{\frac12}
\end{align*}
and
\[
\fint_{B_{\kappa^j r_0}} \abs{\vec a_j \cdot x/2 + b_j}^{\frac12}dx  = \frac{2^d}{\abs{B_{\kappa^j r_0}}}\int_{B_{\kappa^j r_0/2}} \abs{\vec a_j \cdot x + b_j}^{\frac12} dx \leq 2^d\fint_{B_{\kappa^j r_0}} \abs{\vec a_j\cdot x + b_j}^{\frac12}dx.
\]
Thus,
\begin{equation}		\label{eq0213tue}
\abs{b_j} \le C \left(\fint_{B_{\kappa^j r_0}}\abs{\vec a_j \cdot x + b_j }^{\frac12}dx\right)^{2}\le C \left(\fint_{B_{\kappa^j r_0}} \abs{u}^{\frac12}\right)^{2},\quad j=0,1,2,\ldots.
\end{equation}
Since $u$ is a continuous function vanishing at $0$, \eqref{eq0213tue} immediately implies
\begin{equation}			\label{eq0215tue}
\lim_{j \to \infty} b_j =0.
\end{equation}

\subsection*{Estimate of $\vec a_j$}
By the quasi-triangle inequality, we have
\[
\abs{(\vec a_j - \vec a_{j-1})\cdot x + (b_j -b_{j-1})}^{\frac12} \le \abs{u-\vec a_j \cdot x- b_j}^{\frac12} + \abs{u-\vec a_{j-1} \cdot x- b_{j-1}}^{\frac12}.
\]
Taking the average over $B_{\kappa^j r_0}$ and using the fact that $\abs{B_{\kappa^{j-1} r_0}}/ \abs{B_{\kappa^j r_0}} = \kappa^{-d}$, we obtain
\begin{equation}			\label{eq1949sat}
\frac{1}{\kappa^j r_0}
\left(\fint_{B_{\kappa^j r_0}}\abs{(\vec a_j - \vec a_{j-1})\cdot x + (b_j -b_{j-1})}^{\frac12}dx\right)^{2} \le C \varphi(\kappa^j r_0) + C \varphi(\kappa^{j-1} r_0)
\end{equation}
for $j=1,2,\ldots$, where $C=C(d, \lambda, \Lambda, \beta)$.

Next, observe that
\[
\abs{(b_j-b_{j-1})}^{\frac12} = \abs{(\vec a_j -\vec a_{j-1})\cdot x + (b_j-b_{j-1}) - 2\{(\vec a_j -\vec a_{j-1})\cdot x/2 + (b_j-b_{j-1})\}}^{\frac12}.
\]
Using the quasi-triangle inequality, we obtain
\[
\abs{(b_j-b_{j-1})}^{\frac12} \le \abs{(\vec a_j -\vec a_{j-1})\cdot x + (b_j-b_{j-1})}^{\frac12} + 2^{\frac12} \,\abs{(\vec a_j -\vec a_{j-1})\cdot x/2 + (b_j-b_{j-1})}^{\frac12}.
\]
Moreover,
\[
\fint_{B_{\kappa^j r_0}} \abs{(\vec a_j -\vec a_{j-1})\cdot x/2 + (b_j-b_{j-1})}^{\frac12}dx \leq
2^d\fint_{B_{\kappa^j r_0}} \abs{(\vec a_j -\vec a_{j-1})\cdot x + (b_j-b_{j-1})}^{\frac12}dx.
\]
Combining these estimates, we conclude that
\[
\abs{b_j-b_{j-1}} \le C(d) \left(\fint_{B_{\kappa^j r_0}}\abs{(\vec a_j - \vec a_{j-1})\cdot x + (b_j -b_{j-1})}^{\frac12}dx\right)^{2},\quad j=1,2,\ldots.
\]
Substituting this into \eqref{eq1949sat}, we derive
\begin{equation}				\label{eq2247sat}
\frac{1}{\kappa^j r_0} \abs{b_j - b_{j-1}} \le C \varphi(\kappa^j r_0)+C \varphi(\kappa^{j-1} r_0),\quad j=1,2,\ldots.
\end{equation}

On the other hand, note that for any $\vec a \in \bR^d$, by writing $\vec a= \abs{\vec a} \vec e$, where $\vec e$ is a unit vector, we have
\begin{equation}		\label{eq0229tue}
\fint_{B_r} \abs{\vec a \cdot x}^{\frac12} dx= \abs{\vec a}^{\frac12} \fint_{B_r}  \abs{\vec e \cdot x}^{\frac12}dx = \abs{\vec a}^{\frac12} \fint_{B_1} r^{\frac12} \abs{\vec e \cdot x}^{\frac12}dx =C(d) r^{\frac12} \abs{\vec a}^{\frac12}.
\end{equation}

Then, by using \eqref{eq0229tue}, the quasi-triangle inequality, \eqref{eq1949sat}, \eqref{eq2247sat},  \eqref{eq4.38}, and the observation that $M_{j-1}(r_0) \le M_{j}(r_0)$, we obtain
\begin{align}
						\nonumber
\abs{\vec a_j-\vec a_{j-1}} &=\frac{C}{\kappa^j r_0} \left(\fint_{B_{\kappa^j r_0}}\abs{(\vec a_j - \vec a_{j-1})\cdot x}^{\frac12}dx\right)^{2}\\
						\nonumber
& \le \frac{C}{\kappa^j r_0} \left(\fint_{B_{\kappa^j r_0}}\abs{(\vec a_j - \vec a_{j-1})\cdot x + (b_j -b_{j-1})}^{\frac12}dx\right)^{2} + \frac{C}{\kappa^j r_0} \abs{b_j - b_{j-1}}\\
						\nonumber
&\le C \varphi(\kappa^j r_0)+C \varphi(\kappa^{j-1} r_0)\\
								\label{eq2316sat}
&\le \frac{C\kappa^{\beta j}}{r_0}  \fint_{B_{r_0}} \abs{u}+ CM_j(r_0)\left\{\tilde \omega_{\mathbf A}(\kappa^j r_0) + \tilde \omega_{\mathbf A}(\kappa^{j-1} r_0)\right\}.
\end{align}

To estimate $\abs{\vec a_0}$, we proceed similarly to \eqref{eq2316sat} by applying \eqref{eq0218tue} and \eqref{eq0213tue} with $j=0$, and using H\"older's inequality to obtain
\begin{equation}			\label{eq0230tue}
\abs{\vec a_0} \le \frac{C}{r_0} \fint_{B_{r_0}} \abs{u},
\end{equation}
where $C=C(d, \lambda, \Lambda, \beta)$.

For $k>l\ge 0$, we derive from \eqref{eq2316sat} and the definition of $M_{j}(r_0)$ that
\begin{align}
						\nonumber
\abs{\vec a_k -\vec a_l} \le \sum_{j=l}^{k-1}\, \abs{\vec a_{j+1}-\vec a_j} &\le  \sum_{j=l}^{k-1}  \frac{C\kappa^{\beta(j+1)}}{r_0} \fint_{B_{r_0}} \abs{u}+ CM_k(r_0) \sum_{j=l}^{k} \tilde \omega_{\mathbf A}(\kappa^j r_0) \\
						\label{eq0946tue}
&\le  \frac{C \kappa^{\beta(l+1)}}{(1-\kappa^\beta)r_0} \fint_{B_{r_0}} \abs{u}+ CM_k(r_0) \int_0^{\kappa^l r_0} \frac{\tilde \omega_{\mathbf A}(t)}{t}\,dt,
\end{align}
where we used Remark~\ref{rmk1147}.
In particular, by taking $k=j$ and $l=0$ in \eqref{eq0946tue}, and using \eqref{eq0230tue}, we obtain for $j=1,2,\ldots$ that
\begin{equation}			\label{eq0900tue}
\abs{\vec a_j} \le \abs{\vec a_j-\vec a_0} + \abs{\vec a_0} \le \frac{C}{r_0} \fint_{B_{r_0}} \abs{u}+ CM_j(r_0) \int_0^{r_0} \frac{\tilde \omega_{\mathbf A}(t)}{t}\,dt.
\end{equation}

Similarly, we obtain from \eqref{eq2247sat} that for $k>l\ge 0$, we have
\begin{equation}			\label{eq0947tue}
\abs{b_k - b_l} \le C \frac{\kappa^{(\beta+1)(l+1)}}{1-\kappa^{\beta+1}} \fint_{B_{r_0}} \abs{u}+ C \kappa^l r_0 M_k(r_0) \int_0^{\kappa^l r_0} \frac{\tilde \omega_{\mathbf A}(t)}{t}\,dt.
\end{equation}

\subsection*{Estimate for $b_j$}
We shall derive improved estimates for $\abs{b_j}$ using the following lemma, where we set
\[
v(x):=u(x)-\vec a_j\cdot x-b_j.
\]

\begin{lemma}			\label{lem1702sat}
For $0<r \le \frac12$, we have
\[
\sup_{B_{r}}\, \abs{v} \le C \left\{\left( \fint_{B_{2r}} \abs{v}^{\frac12} \right)^{2} + (r \abs{\vec a_j}+b_j) \int_0^r \frac{\omega_{\mathbf A}(t)}{t}\,dt\right\},
\]
where $C=C(d, \lambda, \Lambda, \omega_{\mathbf A})$.
\end{lemma}

\begin{proof}
Denote
\[
\mathbf F(x)=-\mathbf A(x) (\vec a_j\cdot x+b_j).
\]
Since $\dv^2(\mathbf{A} u)=0$ in $B_1$, we have
\[
\dv^2(\mathbf A v) = \dv^2 \mathbf F  \quad\text{in}\quad B_{2r},
\]
for $0<r \le \frac12$.
For $x_0 \in B_{3r/2}$ and $0<t \le r/4$, we decompose $v$ as $v=v_1+v_2$, where $v_1 \in L^{p}(B_t(x_0))$ (for some $p>1$) is the weak solution of the problem
\[
\left\{
\begin{aligned}
\dv^2(\bar{\mathbf A}v_1) &= \dv^2 \left(\mathbf F- \mathbf L-(\mathbf{A}- \bar{\mathbf A})v\right)\;\text{ in }\;B_t(x_0),\\
(\bar{\mathbf A}\nu\cdot \nu)v_1&=\left(\mathbf F-\mathbf L-(\mathbf{A}- \bar{\mathbf A}) v\right)\nu\cdot \nu \;\text{ on }\;\partial B_t(x_0),
\end{aligned}
\right.
\]
where $\bar{\mathbf A}=(\mathbf A)_{B_t(x_0)}$ and
\[
\mathbf L(x)=-\bar{\mathbf A} (\vec a_j\cdot x+b_j).
\]

By \cite[Lemma~2.23]{DK17} and following \cite[pp. 422--423]{DK17}, we obtain
\[
\left(\fint_{B_t(x_0)} \abs{v_1}^{\frac12} \right)^{2}
\le C \left(\fint_{B_t(x_0)}\abs{\mathbf A-\bar{\mathbf A}}\right) (r\abs{\vec a_j}+\abs{b_j})
+C \left(\fint_{B_t(x_0)}\abs{\mathbf A-\bar{\mathbf A}}\right)\norm{v}_{L^\infty(B_t(x_0))},
\]
and thus, we have
\[
\left(\fint_{B_t(x_0)} \abs{v_1}^{\frac12} \right)^{2} \le C \omega_{\mathbf A}(t)(r\abs{\vec a_j}+\abs{b_j})+ C  \omega_{\mathbf A}(t)\norm{v}_{L^\infty(B_t(x_0))}.
\]

On the other hand, observe that $\dv^2 \mathbf L=0$, and thus $v_2=v-v_1$ satisfies
\[
\dv^2(\bar{\mathbf A}v_2) = 0 \quad\text{in}\quad  B_t(x_0).
\]
We note that $v_2$ satisfies the same interior estimate as appearing in the proof of \cite[Theorem 1.10]{DK17}.
The rest of proof remains the same as that of \cite[Theorem 1.10]{DK17}.
\end{proof}

By Lemma \ref{lem1702sat}, \eqref{eq0203tue}, and \eqref{eq4.38}, we have
\begin{multline}				\label{eq0934wed0}
 \norm{v}_{L^\infty(B_{\frac12 \kappa^j r_0})} \le C\kappa^{(1+\beta)j} \fint_{B_{r_0}} \abs{u} + C\kappa^j r_0 M_j(r_0) \tilde\omega_{\mathbf A}(\kappa^j r_0)\\
+ C \left(\kappa^j r_0\abs{\vec a_j}+\abs{b_j}\right) \int_0^{\kappa^j r_0} \frac{\omega_{\mathbf A}(t)}{t}\,dt,\quad j=1,2,\ldots,
\end{multline}
where $C=C(d, \lambda, \Lambda, \omega_{\mathbf A}, \beta)$.

In particular, from \eqref{eq0934wed0}, we infer that
\begin{multline*}
\abs{b_j}=\abs{v(0)} \le C\kappa^{(1+\beta)j} \fint_{B_{r_0}} \abs{u} + C\kappa^j r_0 M_j(r_0) \tilde\omega_{\mathbf A}(\kappa^j r_0)\\
+C \abs{\vec a_j}\kappa^j r_0 \int_0^{\kappa^j r_0} \frac{\omega_{\mathbf A}(t)}{t}\,dt
+C \abs{b_j} \int_0^{\kappa^j r_0} \frac{\omega_{\mathbf A}(t)}{t}\,dt.
\end{multline*}
Let us fix $r_1>0$ so that
\begin{equation}			\label{eq0903tue}
C \int_0^{r_1} \frac{\omega_{\mathbf A}(t)}{t}\,dt\le \frac12.
\end{equation}
Observe that $r_1$ depends solely on $d$, $\lambda$, $\Lambda$, $\omega_{\mathbf A}$, and $\beta$.
Note that we have not yet chosen $r_0 \in (0,\frac12]$.
We shall require $r_0 \le r_1$ so that we have
\[
\abs{b_j} \le  C\kappa^{(1+\beta)j} \fint_{B_{r_0}} \abs{u} + C\kappa^j r_0 M_j(r_0) \tilde\omega_{\mathbf A}(\kappa^j r_0)
+C\abs{\vec a_j}\kappa^j r_0\int_0^{\kappa^j r_0} \frac{\omega_{\mathbf A}(t)}{t}\,dt.
\]
This, together with \eqref{eq0900tue} and Remark~\ref{rmk1147}, yields
\begin{equation}			\label{eq7.51}
\abs{b_j} \le C \kappa^{j} r_0 \left\{\kappa^{\beta j}+\int_0^{\kappa^j r_0} \frac{\omega_{\mathbf A}(t)}{t}\,dt\right\}\frac{1}{r_0} \fint_{B_{r_0}} \abs{u} +C\kappa^j r_0 M_j(r_0) \int_0^{\kappa^j r_0} \frac{\tilde \omega_{\mathbf A}(t)}{t}\,dt.
\end{equation}

\subsection*{Convergence of $\vec a_j$}
By \eqref{eq0934wed0}, \eqref{eq0900tue}, \eqref{eq7.51}, and Remark~\ref{rmk1147}, we have
\begin{multline}				\label{eq1920thu}
\norm{u-\vec a_j\cdot x-b_j}_{L^\infty(B_{\frac12 \kappa^j r_0})}
\le C \kappa^j r_0 \left\{\kappa^{\beta j}+ \int_0^{\kappa^j r_0} \frac{\omega_{\mathbf A}(t)}{t}\,dt \right\}\frac{1}{r_0} \fint_{B_{r_0}} \abs{u}\\
+C\kappa^j r_0 M_j(r_0)  \int_0^{\kappa^j r_0} \frac{\tilde \omega_{\mathbf A}(t)}{t}\,dt.
\end{multline}
Then, from \eqref{eq1920thu}, \eqref{eq0900tue}, \eqref{eq7.51}, and Remark~\ref{rmk1147}, we infer that
\begin{equation}				\label{eq1555sun}
\frac {1}{\kappa^{j} r_0}\norm{u}_{L^\infty(B_{\frac12 \kappa^j r_0})}
\le \frac C {r_0} \fint_{B_{r_0}} \abs{u} + C M_j(r_0) \int_0^{r_0} \frac{\tilde \omega_{\mathbf A}(t)}{t}\,dt,
\end{equation}
where $C=C(d, \lambda, \Lambda, \omega_{\mathbf A}, \beta)$.

\begin{lemma}				\label{lem1548sun}
There exists a sufficiently small number $r_0=r_0(d, \lambda, \Lambda, \omega_{\mathbf A}, \beta) \in (0,\frac12)$ such that
\begin{equation}			\label{eq1951thu}
\sup_{j \ge 1} M_j(r_0) =\sup_{i \ge 0} \frac{1}{\kappa^i r_0} \,\norm{u}_{L^\infty(B_{\kappa^i r_0})} \le \frac{C}{r_0} \fint_{B_{2r_0}} \abs{u},
\end{equation}
where $C=C(d, \lambda, \Lambda, \omega_{\mathbf A}, \beta)$.
\end{lemma}

\begin{proof}
We shall denote
\[
\quad c_i= \frac{1}{\kappa^i r_0} \norm{u}_{L^\infty(B_{\kappa^i r_0})},\quad i=0,1,2,\ldots.
\]
Then, by \eqref{eq1244sat}, we clearly have
\begin{equation}			\label{eq0906tue}
M_1(r_0)=c_0\quad\text{and}\quad
M_{j+1}(r_0)=\max(M_j(r_0), c_j),\quad j=1,2,\ldots.
\end{equation}
Recall that we have already chosen $\kappa=\kappa(d, \lambda, \Lambda, \beta) \in (0,\frac12)$, and thus, from \eqref{eq1555sun}, we conclude that there exists a constant $C=C(d, \lambda, \Lambda, \omega_{\mathbf A}, \beta)>0$ such that
\begin{equation}			\label{eq0848tue}
c_{j+1} \le C \left\{\frac{1}{r_0} \fint_{B_{r_0}}\abs{u}+ M_j(r_0) \int_0^{r_0} \frac{\tilde \omega_{\mathbf A}(t)}{t}\,dt \right\},\quad j=1,2,\ldots.
\end{equation}
By \cite[Theorem 1.10]{DK17}, we have
\begin{equation}			\label{eq0849tue}
c_1 =\frac{1}{\kappa r_0} \norm{u}_{L^\infty(B_{\kappa r_0})} \le \frac{1}{\kappa r_0} \norm{u}_{L^\infty(B_{r_0})} =\frac{1}{\kappa} c_0  \le \frac{C}{r_0} \fint_{B_{2r_0}} \abs{u},
\end{equation}
where $C=C(d, \lambda, \Lambda, \omega_{\mathbf A},\beta)$.

From \eqref{eq0848tue} and \eqref{eq0849tue}, we see that there is $\gamma=\gamma(d, \lambda, \Lambda,\omega_{\mathbf A}, \beta)>0$ such that
\[
c_0,\, c_1 \le \frac{\gamma}{r_0} \fint_{B_{2r_0}}\abs{u} \quad \text{and}\quad c_{j+1} \le \gamma \left\{\frac{1}{r_0} \fint_{B_{2r_0}}\abs{u}+ M_j(r_0) \int_0^{r_0} \frac{\tilde \omega_{\mathbf A}(t)}{t}\,dt \right\},\;\; j=1,2,\ldots.
\]

Now, we fix a number $r_0 \in (0,\frac12]$ such that
\[
\gamma \int_0^{r_0} \frac{\tilde \omega_{\mathbf A}(t)}{t}\,dt \le \frac12
\]
and also $r_0 \le r_1$, where $r_1$ is in \eqref{eq0903tue}.
Then, we have
\begin{equation}			\label{eq2114mon}
c_0, c_1 \le \frac{\gamma}{r_0} \fint_{B_{2r_0}}\abs{u} \quad \text{and}\quad c_{j+1} \le \frac{\gamma}{r_0} \fint_{B_{2r_0}}\abs{u} +\frac12 M_j(r_0),\;\; j=1,2,\ldots.
\end{equation}

By induction, it follows form \eqref{eq0906tue} and \eqref{eq2114mon} that
\[
c_{2k},\,c_{2k+1},\,M_{2k+1}(r_0),\,M_{2k+2}(r_0) \le \frac{\gamma}{r_0} \fint_{B_{2r_0}}\abs{u}\, \sum_{i=0}^{k} \frac{1}{2^i},\quad k=0,1,2,\ldots,
\]
and the lemma follows.
\end{proof}

\begin{remark}			\label{rmk1203tue}
The proof of the previous lemma actually shows that for $0<R \le r_0$, we have
\[
\sup_{0<r\le R} \,\frac{1}{r}  \norm{u}_{L^\infty(B_r)} \le \frac{C}{R} \fint_{B_{2R}} \abs{u},
\]
where $C=C(d, \lambda, \Lambda, \omega_{\mathbf A})$.
To see this, we can fix $\kappa \in (0,\frac12)$ corresponding to some $\beta\in (0,1)$, say $\beta=\frac12$, repeat the same argument with $R$ in place of $r_0$, and use the following inequality:
\[
\frac{1}{\kappa^{j}R} \norm{u}_{L^\infty(B_{\kappa^{j+1} R})} \le \sup_{\kappa^{j+1}R \le r \le \kappa^j R}  \frac{1}{r} \norm{u}_{L^\infty(B_{r})} \le \frac{1}{\kappa^{j+1}R} \norm{u}_{L^\infty(B_{\kappa^j R})}.
\]
\end{remark}

Now, Lemma \ref{lem1548sun} and \eqref{eq0946tue} imply that the sequence $\set{\vec a_j}$ is a Cauchy sequence in $\bR^d$, and thus $\vec a_j \to \hat{\vec a}$ for some $\hat{\vec a} \in \bR^d$.
Moreover, by taking the limit as $k\to \infty$ in \eqref{eq0946tue} and \eqref{eq0947tue} (while recalling \eqref{eq1951thu} and \eqref{eq0215tue}), respectively, and then setting $l=j$, we obtain the following estimates:
\begin{equation}	\label{eq1806sun}
\begin{aligned}
\abs{\vec a_j-\hat{\vec a}} &\le C \left\{ \kappa^{\beta j}+ \int_0^{\kappa^j r_0} \frac{\tilde \omega_{\mathbf A}(t)}{t}\,dt\right\}\frac{1}{r_0} \fint_{B_{2r_0}} \abs{u},\\
\abs{b_j}  &\le C \kappa^j \left\{\kappa^{\beta j} +\int_0^{\kappa^j r_0} \frac{\tilde \omega_{\mathbf A}(t)}{t}\,dt \right\}\fint_{B_{2r_0}} \abs{u}.
\end{aligned}
\end{equation}

By the triangle inequality, \eqref{eq1920thu}, \eqref{eq1806sun}, and Remark~\ref{rmk1147}, we obtain
\begin{align}			\nonumber
\norm{u-\hat{\vec a}\cdot x}_{L^\infty(B_{\frac12 \kappa^j r_0})}& \le \norm{u- \vec a_j \cdot x - b_j}_{L^\infty(B_{\frac12 \kappa^j r_0})} + \frac{\kappa^j r_0}{2} \abs{\vec a_j - \hat{\vec a}} +  \abs{b_j}\\
					\label{eq2221sun}
&\le C \kappa^j r_0\left\{\kappa^{\beta j}+\int_0^{\kappa^j r_0} \frac{\tilde\omega_{\mathbf A}(t)}{t}\,dt\right\} \frac{1}{r_0} \fint_{B_{2r_0}} \abs{u}.
\end{align}

\subsection*{Conclusion}
It follows from \eqref{eq2221sun} that
\begin{equation}		\label{eq1036wed}
\frac{1}{r} \norm{u-\hat{\vec a}\cdot x}_{L^\infty(B_r)}  \le \varrho_{\mathbf A}(r) \left( \frac{1}{r_0} \fint_{B_{2r_0}} \abs{u} \right),
\end{equation}
where
\begin{equation}			\label{eq2143sat}
\varrho_{\mathbf A}(r)=C\left\{\left(\frac{2r}{\kappa r_0}\right)^\beta+\int_0^{2r/\kappa} \frac{\tilde\omega_{\mathbf A}(t)}{t}\,dt\right\}.
\end{equation}
Note that $\varrho_{\mathbf A}$ is a modulus of continuity  determined by $d$, $\lambda$, $\Lambda$, $\omega_{\mathbf A}$, and $\beta \in (0,1)$.

In particular, we conclude from \eqref{eq1036wed} that $u$ is differentiable at $0$.
Moreover, it follows from \eqref{eq0900tue} and \eqref{eq1951thu} that (noting that $Du(0)=\hat{\vec a}=\lim_{j\to \infty} \vec a_j$) we have
\[
\abs{D u (0)} \le \frac{C}{r_0} \fint_{B_{2r_0}} \abs{u},
\]
where $C=C(d, \lambda, \Lambda, \omega_{\mathbf A},\beta)$. \qed

\begin{remark}
From \eqref{eq2143sat}, we observe that when $\mathbf A \in C^{\alpha}$ for some $\alpha \in (0,1)$, we have $\varrho_{\mathbf A}(r) \lesssim r^\alpha$ by choosing $\beta=\alpha$.
This leads to a pointwise $C^{1,\alpha}$ estimate.
\end{remark}

\section{Proof of Theorem~\ref{thm-main01}: General case}	\label{sec3}
We now proceed to prove Theorem~\ref{thm-main01} in the general setting.
Let $u$ be a locally bounded function that satisfies
\[
L^* u=\dv^2 \mathbf{f}+\dv \vec g+ h\quad\text{in }\;\Omega.
\]

By \cite[Theorem 1.8]{DEK18}, we know that $u$ is continuous in $\Omega$ (see also \cite[Proposition 2.21]{DEK18}). Suppose that $u(x^o) = 0$ for some $x^o \in \Omega$.
As in Section \ref{sec2}, we can assume that $x^o = 0$, and that  $d_{x^o} \geq 2$.
This implies that $B_2(0) \subset \Omega$ (see Remark~\ref{rmk_rescaling}).

For $0<r \le \frac12$, we define $\widehat {\mathbf f}=(\widehat f^{ij})$ by
\[
\widehat f^{ij}(x)=f^{ij}(x)-{(f^{ij})}_{B_r}-{({D f^{ij}})}_{B_r}\cdot x.
\]
Note that $\int_{B_r} \widehat f^{ij} =0$.
By the Poincar\'e inequality, we then have
\begin{equation}			\label{eq1103fri}
\fint_{B_r} \abs{\widehat f^{ij}} \le C r \fint_{B_r} \abs{D \widehat f^{ij}} = C r \fint_{B_r} \abs{D f^{ij}- ({D f^{ij}})_{B_r}} \le C r \omega_{D \mathbf f}(r).
\end{equation}

Denote by $\bar{\mathbf A}$ and $\bar{\vec g}$ the averages of $\mathbf A$ and $\vec g$ over the ball $B_r$, respectively.
We decompose $u$ as $u=v+w$, where $w$ is the $L^p$ weak solution (for some $p>1$) of the problem
\[
\left\{
\begin{aligned}
\dv^2( \bar{\mathbf A} w) &= \dv^2 (\widehat{\mathbf f}-(\mathbf{A}-\bar{\mathbf A})u)+\dv (\vec b u + \vec g -\bar{\vec g})+h-cu\;\mbox{ in }\; B_r,\\
(\bar{\mathbf A} \nu \cdot \nu) w&= (\widehat{\mathbf f}-(\mathbf{A}-\bar{\mathbf A})u) \nu \cdot \nu \;\mbox{ on }\;\partial B_r.
\end{aligned}
\right.
\]

By Lemma \ref{lem01} below, which is an extension of \cite[Lemma 2.23]{DK17}, we obtain the following estimate via rescaling:
\begin{multline}			\label{eq1020fri}
\left(\fint_{B_r} \abs{w}^{\frac12}\,dx\right)^{2} \le C r \omega_{D \mathbf f}(r)+C\omega_{\mathbf A}(r) \norm{u}_{L^\infty(B_r)}+C\left(r \fint_{B_r} \abs{\vec b}\right)\norm{u}_{L^\infty(B_r)}\\
+ Cr\omega_{\vec g}(r) + Cr^2 \fint_{B_r} \abs{h} + C \left(r^2 \fint_{B_r} \abs{c}\right) \norm{u}_{L^\infty(B_r)},
\end{multline}
where \eqref{eq1103fri} has been used.

\begin{lemma}			\label{lem01}
Let $B=B_1(0)$.
Let $\mathbf A_0$ be a constant symmetric matrix satisfying the condition \eqref{ellipticity-nd}.
For $\mathbf f \in L^p(B)$, $\vec g \in L^p(B)$, and $h \in L^p(B)$, for some $p>1$, let $u \in L^p(B)$ be the weak solution to the problem
\[
\left\{
\begin{aligned}
\dv^2(\mathbf A_0 u)&= \dv^2 \mathbf f+ \dv \vec g + h\;\mbox{ in }\; B,\\
(\mathbf A_0 \nu \cdot \nu) u &= \mathbf f \nu\cdot \nu \;\mbox{ on } \; \partial B.
\end{aligned}
\right.
\]
Then for any $t>0$, we have
\[
\Abs{\set{x \in B : \abs{u(x)} > t}}  \le \frac{C}{t} \left(\int_{B} \abs{\mathbf f}+ \abs{\vec g}+ \abs{h} \right),
\]
where $C=C(d, \lambda, \Lambda)$.
\end{lemma}

\begin{proof}
See Appendix.
\end{proof}

Since $\dv^2 \widehat{\mathbf f}=\dv^2 \mathbf f$ and $\dv (\vec g -\bar{\vec g})=\dv \vec g$, we see that $v=u-w$ satisfies
\[
\dv^2(\bar{\mathbf A} v)=\tr(\bar{\mathbf A} D^2 v)=0\;\text { in }\;B_r.
\]

Let $\varphi(r)$ be as defined in \eqref{eq1100sat}, and let $\beta \in (0,1)$ be an arbitrary but fixed number.
Given $\beta$, we choose a number $\kappa \in (0, \frac12)$ such that $C \kappa \le  \kappa^{\beta}$.

Let $\omega_{\rm coef}(r)$ and $\omega_{\rm dat}(r)$ be as defined in \eqref{omega_coef} and \eqref{omega_dat}.
By Remark~\ref{rmk_omega}, we observe that both $\omega_{\rm coef}(r)$ and $\omega_{\rm dat}(r)$ satisfy the Dini condition.

By using \eqref{eq1020fri} instead of \eqref{eq1622thu} and following the same computations as in \eqref{eq0224sat}, we obtain, similar to \eqref{eq1110sat}, the estimate
\[
\varphi(\kappa r) \le \kappa^\beta \varphi(r) + C \omega_{\rm coef}(r)\, \frac{1}{r} \norm{u}_{L^\infty(B_r)} +C \omega_{\rm dat}(r),
\]
where $C=C(d, \lambda, \Lambda, \beta)$.

Let $M_j(r_0)$ be as defined in \eqref{eq1244sat}, where $r_0\in (0, \frac{1}{2}]$ is a number to be chosen later.
Define $\tilde \omega_{\rm coef}$ and $\tilde \omega_{\rm dat}$ as in \eqref{eq1245sat}, with $\omega_{\rm coef}$ and $\omega_{\rm dat}$ replacing $\omega_{\mathbf A}$, respectively.

Then, by replicating the same argument as in Section~\ref{sec2}, we obtain, similar to \eqref{eq4.38}, the following estimate:
\begin{equation}			\label{eq1120sat}
\varphi(\kappa^j r_0) \le  \frac{\kappa^{\beta j}}{r_0} \fint_{B_{r_0}} \abs{u}+ CM_j(r_0) \tilde \omega_{\rm coef}(\kappa^j r_0) +C \tilde \omega_{\rm dat}(\kappa^j r_0).
\end{equation}

For $j=0,1,2,\ldots$, let $\vec a_j \in \bR^d$ and $b_j \in \bR$ be chosen as in \eqref{eq0203tue}.
Then, by using \eqref{eq1120sat} instead of \eqref{eq4.38} and replicating the same argument as in Section~\ref{sec2}, we obtain the following conclusion:
\begin{equation}			\label{eq1717sun}
\lim_{j\to \infty} b_j=0.
\end{equation}

Additionally, by repeating the same computation that leads to \eqref{eq0946tue} and \eqref{eq0947tue}, we obtain the following estimate:
\begin{multline}			\label{eq1718sun}
\abs{\vec a_k -\vec a_l} +\frac{\abs{b_k - b_l}}{\kappa^l r_0}
\le \frac{C \kappa^{\beta l}}{r_0} \fint_{B_{r_0}} \abs{u}+ CM_k(r_0) \int_0^{\kappa^l r_0} \frac{\tilde \omega_{\rm coef}(t)}{t}\,dt
\\+C \int_0^{\kappa^l r_0} \frac{\tilde \omega_{\rm dat}(t)}{t}\,dt,\quad \text{for }\;k>l\ge 0,
\end{multline}
where $C=C(d, \lambda, \Lambda,\beta)$.
In particular, similar to \eqref{eq0900tue}, we obtain the following bound for $\vec a_j$:
\begin{equation}			\label{eq1628sun}
\abs{\vec a_j} \le  \frac{C}{r_0} \fint_{B_{r_0}} \abs{u}+ CM_j(r_0) \int_0^{r_0} \frac{\tilde \omega_{\rm coef}(t)}{t}\,dt+C \int_0^{r_0} \frac{\tilde \omega_{\rm dat}(t)}{t}\,dt,
\end{equation}
where $C=C(d, \lambda, \Lambda, \beta)$.

Next, for $j=0,1,2,\ldots$, define
\[
v(x):=u(x)-\vec a_j\cdot x-b_j.
\]
We observe that the operator $L^*$ applied to $v$ gives:
\begin{align*}
L^{*} v&= L^* u - L^*(\vec a_j \cdot x + b_j)\\
&=\dv^2 (\mathbf{f} -\mathbf A (\vec a_j \cdot x + b_j))+\dv(\vec g+ (\vec a_j \cdot x + b_j) \vec b) + h-(\vec a_j \cdot x + b_j)c.
\end{align*}

\begin{lemma}			\label{lem1146sat}
For $0<r \le \frac12$, we have
\[
\sup_{B_{r}}\, \abs{v} \le C \left\{ \left( \fint_{B_{2r}} \abs{v}^{\frac12} \right)^{2} + r\int_0^r \frac{\omega_{\rm dat}(t)}{t}\,dt +(r\abs{\vec a_j}+\abs{b_j}) \int_0^r \frac{\omega_{\rm coef}(t)}{t}\,dt \right\},
\]
where $C=C(d, \lambda, \Lambda, \omega_{\rm coef})$.
\end{lemma}
\begin{proof}
The proof closely follows \cite[Lemma 2.2]{KL21}, which is itself derived from the proof of \cite[Theorem 1.8]{DEK18}.
As in the proof of \cite[Lemma 2.2]{KL21}, we aim to control the following quantity for $x_0 \in B_{3r/2}$ and $0<t\le r/4$:
\[
\phi(x_0, t):= \inf_{q \in \bR} \left(\fint_{B_t(x_0)} \abs{v-q}^{\frac12} \right)^2.
\]
Let the following notations be introduced:
\[
\mathbf F= \mathbf{f} -\mathbf A (\vec a_j \cdot x + b_j),\quad
\vec G=\vec g+ (\vec a_j \cdot x + b_j) \vec b,\quad
 H=h-(\vec a_j \cdot x + b_j)c.
\]
For $x_0 \in B_{3r/2}$ and $0<t \le r/4$, let $\bar{\mathbf A}$ and $\bar{\vec g}$ denote the averages of $\mathbf A$ and $\vec g$ over the ball $B_t(x_0)$, respectively.
Now, define $\mathbf L=(\ell^{kl})$ as follows:
\[
\ell^{kl}(x)={(D f^{kl})}_{B_t(x_0)}\cdot (x-x_0) +{(f^{kl})}_{B_t(x_0)}-{(a^{kl})}_{B_t(x_0)} (\vec a_j \cdot x+b_j).
\]

We decompose $v$ as $v=v_1+v_2$, where $v_1 \in L^{p}(B_t(x_0))$, for some $p>1$, is the solution of the problem:
\[
\left\{
\begin{aligned}
\dv^2(\bar{\mathbf A}v_1) &= \dv^2 \left(\mathbf{F} -\mathbf{L}-(\mathbf{A}- \bar{\mathbf A})v\right)+\dv(\vec G-\bar{\vec g} +\vec b v)+H-cv\;\mbox{ in }\;B_t(x_0),\\
(\bar{\mathbf A}\nu\cdot \nu)v_1&=\left(\mathbf{F}-\mathbf{L}-(\mathbf{A}- \bar{\mathbf A}) v\right)\nu\cdot \nu \;\mbox{ on }\;\partial B_t(x_0).
\end{aligned}
\right.
\]
Then, by Lemma~\ref{lem01} via rescaling, we obtain the following estimate:
\begin{multline}				\label{eq2025mon}
\left(\fint_{B_t(x_0)} \abs{v_1}^{\frac12} \right)^{2} \lesssim \fint_{B_t(x_0)}\abs{\mathbf F-\mathbf L}+ \left(\fint_{B_t(x_0)}\abs{\mathbf A-\bar{\mathbf A}}\right)\norm{v}_{L^\infty(B_t(x_0))}+t \fint_{B_t(x_0)} \abs{\vec G-\bar{\vec g}}\\
+ \left(t \fint_{B_t(x_0)}\abs{\vec b}\right)\norm{v}_{L^\infty(B_t(x_0))} + t^2 \fint_{B_t(x_0)} \abs{H} +  \left(t^2 \fint_{B_t(x_0)}\abs{c}\right)\norm{v}_{L^\infty(B_t(x_0))}.
\end{multline}
By definitions of $\mathbf F$, $\vec G$, and $H$, and applying the Poincar\'e inequality (cf. \eqref{eq1103fri}), we obtain:
\[
\begin{aligned}
\fint_{B_t(x_0)}\abs{\mathbf F-\mathbf L} &\le Ct \fint_{B_t(x_0)}\abs{D\mathbf f-(D\mathbf f)_{B_t(x_0)}}+(2r \abs{\vec a_j}+\abs{b_j}) \fint_{B_t(x_0)}\abs{\mathbf A-\bar{\mathbf A}},\\
t\fint_{B_t(x_0)}\abs{\vec G-\bar{\vec g}} &\le t \fint_{B_t(x_0)}\abs{\vec g-\bar{\vec g}}+ (2r \abs{\vec a_j}+\abs{b_j}) t\fint_{B_t(x_0)}\abs{\vec b},\\
t^2\fint_{B_t(x)}\Abs{H} &\le t^2 \fint_{B_t(x)}\abs{h}+ (2r \abs{\vec a_j}+\abs{b_j}) t^2\fint_{B_t(x)}\abs{c}.
\end{aligned}
\]
Thus, using the definitions \eqref{omega_coef} and \eqref{omega_dat}, we deduce from \eqref{eq2025mon} that
\begin{equation}				\label{eq2026mon}
\left(\fint_{B_t(x_0)} \abs{v_1}^{\frac12} \right)^{2} \le C \omega_{\rm coef}(t)\norm{v}_{L^\infty(B_t(x_0))} + Cr \omega_{\rm dat}(t) + C(r\abs{\vec a_j}+\abs{b_j}) \omega_{\rm coef}(t),
\end{equation}
where $C=C(d, \lambda, \Lambda)$.
We note that \eqref{eq2026mon} corresponds to (A.1) in \cite{KL21}.

On the other hand, since $\dv^2 \mathbf{L}=0$ and $\dv \bar{\vec g}=0$, it follows that $v_2=v-v_1$ satisfies
\[
\dv^2(\bar{\mathbf A}v_2) = 0 \;\;\text{ in }\;  B_t(x_0).
\]
Thus, $v_2$ satisfies the estimate corresponding to (A.2) in \cite{KL21}.
The remainder of the proof follows the same arguments as those in \cite[Lemma 2.2]{KL21}.
\end{proof}

By using Lemma~\ref{lem1146sat} and \eqref{eq1120sat}, we obtain the following estimate:
\begin{multline}				\label{eq1218mon}
\norm{v}_{L^\infty(B_{\frac12 \kappa^j r_0})} \le C\kappa^{(1+\beta)j} \fint_{B_{r_0}} \abs{u}+C\kappa^j r_0 M_j(r_0) \tilde\omega_{\rm coef}(\kappa^j r_0)
+C \kappa^j r_0 \tilde \omega_{\rm dat}(\kappa^j r_0)\\
+C \left(\abs{\vec a_j}\kappa^j r_0+\abs{b_j}\right) \int_0^{\kappa^j r_0} \frac{\omega_{\rm coef}(t)}{t}\,dt + C \kappa^j r_0 \int_0^{\kappa^j r_0} \frac{\omega_{\rm dat}(t)}{t}\,dt.
\end{multline}
We require $r_0 \in (0,\frac12]$ to satisfy $r_0 \le r_1$, where $r_1= r_1(d, \lambda, \Lambda, \omega_{\rm coef}, \beta)>0$ is chosen so that
\[
C \int_0^{r_1} \frac{\omega_{\rm coef}(t)}{t}\,dt\le \frac12.
\]
Note that $\abs{v(0)}=\abs{b_j}$.
Similarly to \eqref{eq7.51}, we derive from \eqref{eq1218mon} and \eqref{eq1628sun} the following inequality:
\begin{multline}			\label{eq1123mon}
\abs{b_j} \le C \kappa^{j}r_0 \left\{\kappa^{\beta j}+\int_0^{\kappa^j r_0} \frac{\omega_{\rm coef}(t)}{t}\,dt\right\}\frac{1}{r_0} \fint_{B_{r_0}} \abs{u}\\+C\kappa^j r_0 M_j(r_0)\int_0^{\kappa^j r_0}\frac{\tilde\omega_{\rm coef}(t)}{t}\,dt
+ C \kappa^j r_0 \int_0^{r_0}\frac{\tilde\omega_{\rm dat}(t)}{t}\,dt,
\end{multline}
where $C=C(d, \lambda, \Lambda,\omega_{\rm coef}, \beta)$.

Then, similar to \eqref{eq1920thu}, we obtain from \eqref{eq1218mon}, \eqref{eq1628sun}, and \eqref{eq1123mon} (recalling $v(x)=u(x)-\vec a_j\cdot x-b_j$) the following estimate:
\begin{align}				\nonumber
&\norm{u-\vec a_j\cdot x-b_j}_{L^\infty(B_{\frac12 \kappa^j r_0})}  \le C \kappa^j r_0 \left\{\kappa^{\beta j}+ \int_0^{\kappa^j r_0} \frac{\omega_{\rm coef}(t)}{t}\,dt \right\} \frac{1}{r_0} \fint_{B_{r_0}} \abs{u}\\
						\nonumber
&\qquad+C \kappa^j r_0 M_j(r_0) \int_0^{\kappa^j r_0} \frac{\tilde\omega_{\rm coef}(t)}{t}\,dt\\
						\label{eq1720sun}
&\qquad +C \kappa^j r_0 \left\{\int_0^{\kappa^j r_0} \frac{\tilde \omega_{\rm dat}(t)}{t}\,dt +\left(\int_0^{r_0} \frac{\tilde \omega_{\rm dat}(t)}{t}\,dt\right)\int_0^{\kappa^j r_0} \frac{\omega_{\rm coef}(t)}{t}\,dt\right\}.
\end{align}
Additionally, similar to \eqref{eq1555sun}, we obtain
\begin{equation}				\label{eq1244mon}
\frac 2 {\kappa^j r_0}\norm{u}_{L^\infty(B_{\frac12 \kappa^j r_0})}
\le \frac C {r_0} \fint_{B_{r_0}} \abs{u} + C M_j(r_0) \int_0^{r_0} \frac{\tilde \omega_{\rm coef}(t)}{t}\,dt + C \int_0^{r_0}\frac{\tilde\omega_{\rm dat}(t)}{t}\,dt,
\end{equation}
where $C=C(d, \lambda, \Lambda, \omega_{\rm coef}, \beta)$.

Then, by following the same proof as in Lemma~\ref{lem1548sun}, we conclude from \eqref{eq1244mon} that there exists $r_0=r_0(d, \lambda, \Lambda, \omega_{\rm coef}, \beta) \in (0,\frac12]$ such that
\begin{equation}			\label{eq1706sun}
\sup_{j \ge 1} M_j(r_0) \le \frac{C}{r_0} \fint_{B_{2r_0}} \abs{u}+C \int_0^{r_0} \frac{\tilde \omega_{\rm dat}(t)}{t}\,dt,
\end{equation}
where $C=C(d, \lambda, \Lambda, \omega_{\rm coef}, \beta)$.

It follows from \eqref{eq1718sun} and \eqref{eq1706sun} that the sequence $\set{\vec a_j}$ converges to some $\hat{\vec a} \in\bR^d$, since it is a Cauchy sequence in $\bR^d$.
By taking the limit as $k\to \infty$ in \eqref{eq1718sun} (while recalling \eqref{eq1706sun} and \eqref{eq1717sun}), and then setting $l=j$, we obtain the following estimate:
\begin{multline}			\label{eq2203sun}
\abs{\vec a_j -\hat{\vec a}} + \frac{\abs{b_j}}{\kappa^j r_0}  \le C \left( \kappa^{\beta j}+\int_0^{\kappa^j r_0} \frac{\tilde \omega_{\rm coef}(t)}{t}\,dt \right) \frac{1}{r_0} \fint_{B_{2r_0}} \abs{u}\\
+C \left(\int_0^{r_0} \frac{\tilde \omega_{\rm dat}(t)}{t}\,dt \right) \int_0^{\kappa^j r_0} \frac{\tilde \omega_{\rm coef}(t)}{t}\,dt+ C \int_0^{\kappa^j r_0} \frac{\tilde \omega_{\rm dat}(t)}{t}\,dt,
\end{multline}
where $C=C(d, \lambda, \Lambda,\omega_{\rm coef}, \beta)$.

Then, similar to \eqref{eq2221sun}, it follows from \eqref{eq1720sun}, \eqref{eq1706sun}, and \eqref{eq2203sun} that
\begin{multline*}
\norm{u-\hat{\vec a}\cdot x}_{L^\infty(B_{\frac12 \kappa^j r_0})} \le
C \kappa^j r_0 \left\{\kappa^{\beta j}+ \int_0^{\kappa^j r_0} \frac{\tilde \omega_{\rm coef}(t)}{t}\,dt \right\}\frac{1}{r_0} \fint_{B_{2r_0}} \abs{u}\\
+C \kappa^j r_0 \left(\int_0^{r_0} \frac{\tilde \omega_{\rm dat}(t)}{t}\,dt\right)\int_0^{\kappa^j r_0} \frac{\tilde\omega_{\rm coef}(t)}{t}\,dt+ C \kappa^j r_0 \int_0^{\kappa^j r_0} \frac{\tilde \omega_{\rm dat}(t)}{t}\,dt.
\end{multline*}

From the previous inequality, we derive the following estimate:
\begin{equation}		\label{eq2010sun}
\frac{1}{r} \norm{u-\hat{\vec a}\cdot x}_{L^\infty(B_r)}  \le \varrho_{\rm coef}(r) \left\{\frac{1}{r_0} \fint_{B_{2r_0}} \abs{u} + \varrho_{\rm dat}(r_0)\right\}+\varrho_{\rm dat}(r),\quad \forall r \in (0,r_0/2),
\end{equation}
where the moduli of continuity $\varrho_{\rm coef}(\cdot)$ and $\varrho_{\rm dat}(\cdot)$ are defined by
\begin{equation}		\label{eq2011sun}
\begin{aligned}
\varrho_{\rm coef}(r)&:=C\left\{\left(\frac{2r}{\kappa r_0}\right)^\beta+\int_0^{2r/\kappa}\frac{\tilde\omega_{\rm coef}(t)}{t}\,dt\right\},\\
\varrho_{\rm dat}(r)&:=C \int_0^{2r/\kappa} \frac{\tilde \omega_{\rm dat}(t)}{t}\,dt,
\end{aligned}
\end{equation}
where $C=C(d, \lambda, \Lambda,\omega_{\rm coef}, \beta)$.

In particular, we conclude from \eqref{eq2010sun} that $u$ is differentiable at $0$.
Moreover, it follows from \eqref{eq1628sun} and \eqref{eq1706sun} that (noting that $Du(0)=\hat{\vec a}=\lim_{j\to \infty} \vec a_j$) we have
\[
\abs{D u (0)} \le C \left(\frac{1}{r_0} \fint_{B_{2r_0}} \abs{u}+ \int_0^{r_0} \frac{\tilde \omega_{\rm dat}(t)}{t}\,dt\right),
\]
where $C=C(d, \lambda, \Lambda, \omega_{\rm coef},\beta)$.
\qed

\section{Proof of Theorem~\ref{thm-main02}}		\label{sec4}

Throughout this section, we use the following notation:
\[
B_r^+=B_r^{+}(0)=B_r(0) \cap \set{x_d >0}\quad\text{and}\quad  T_r=T_r(0)=B_r(0) \cap \set{x_d=0}.
\]
We fix a smooth (convex) domain $\cD$ satisfying $B_{1/2}^{+}(0) \subset \cD \subset B_1^{+}(0)$, such that $\partial \cD$ contains a flat portion $T_{1/2}(0)$.
For $x_0 \in \partial \bR^d_{+}=\set{x_d=0}$, we define
\[
B_r^{+}(x_0)=B_r^{+}(0)+ x_0,\quad T_r(x_0)=T_r(0)+x_0,\quad\text{and}\quad \cD_r(x_0)= r \cD+ x_0.
\]

\subsection{Special case:  $\partial\Omega$ is a flat near $x^o$.}		\label{sec:flat}
We begin by considering the simple case where $x^o$ lies on a flat portion of the boundary.
By Remark~\ref{rmk_rescaling}, we may assume that $x^o=0$, $B_2^{+} \subset \Omega$, and $T_2  \subset \partial\Omega$, such that $u=f^{dd}/a^{dd}$ on $T_2$.

By \cite[Theorem 1.8]{DEK18}, we have $u \in C(\overline {B_1^+})$.
We now show that $u$ is differentiable at $0$.
Denote by $\bar{\mathbf A}$ and $\bar{\vec g}$ the averages of $\mathbf A$ and $\vec g$ over the half ball $B_{2r}^+$, respectively, and let $\widehat{\mathbf f}=(\widehat f^{ij})$ be defined by
\[
\widehat f^{ij}(x)=f^{ij}(x)-(D f^{ij})_{B_{2r}^+} \cdot x-m^{ij},
\]
where the constants $m^{ij}$ are chosen so that $\int_{B_{2r}^+} \widehat f^{ij}=0$.
By the Poincar\'e inequality, we have
\begin{equation}			\label{eq1025mon}
\fint_{B_{2r}^+} \abs{\widehat f^{ij}} \le C r \fint_{B_{2r}^+} \abs{D\widehat f^{ij}} = C r \fint_{B_{2r}^+} \abs{D f^{ij}- (Df^{ij})_{B_{2r}^+}}  \le C r \omega_{D \mathbf f}(2r).
\end{equation}
We decompose $u$ as $u=v+w$, where $w$ is the $L^p$ weak solution (for some $p>1$) of the problem
\begin{equation}			\label{eq1009sat}
\left\{
\begin{aligned}
\dv^2( \bar{\mathbf A} w) &= \dv^2 (\widehat{\mathbf f}-(\mathbf{A}-\bar{\mathbf A})u)+\dv (\vec b u + \vec g -\bar{\vec g})+h-cu\;\mbox{ in }\; \cD_{2r},\\
(\bar{\mathbf A} \nu \cdot \nu) w&= (\widehat{\mathbf f}-(\mathbf{A}-\bar{\mathbf A})u) \nu \cdot \nu \;\mbox{ on }\;\partial \cD_{2r}.
\end{aligned}
\right.
\end{equation}
We note that Lemma \ref{lem01} remains unchanged if we replace $B_1$ by $\cD$.
Therefore, by the geometry of $\cD$ and rescaling, we have
\begin{multline}			\label{eq1649sat}
\left(\fint_{\cD_{2r}} \abs{w}^{\frac12}\right)^{2} \le  Cr \omega_{D \mathbf f}(2r)+C\omega_{\mathbf A}(2r)\,\norm{u}_{L^\infty(B_{2r}^+)}+Cr\left(\fint_{B^+_{2r}} \abs{\vec b}\right) \norm{u}_{L^\infty(B_{2r}^+)}\\
+Cr \omega_{\vec g}(2r)+ Cr^2 \fint_{B_{2r}^+} \abs{h} + Cr^2 \left(\fint_{B^+_{2r}} \abs{c} \right) \norm{u}_{L^\infty(B_{2r}^+)},
\end{multline}
where we used \eqref{eq1025mon}.

Let $\omega_{\rm coef}$ and $\omega_{\rm dat}$ be defined as in \eqref{omega_coef} and \eqref{omega_dat}.
Then, \eqref{eq1649sat} implies
\begin{equation}			\label{eq2017fri}
\left(\fint_{B_{r}^+} \abs{w}^{\frac12}\right)^{2} \le C \omega_{\rm coef}(2r)\, \norm{u}_{L^\infty(B_{2r}^+)}+ C r\omega_{\rm dat}(2r).
\end{equation}
Note that $v=u-w$ satisfies
\[
\dv^2(\bar{\mathbf A} v)=0\;\text { in }\; B_{r}^+,\quad
(\bar{\mathbf A} \nu \cdot \nu)v=(\mathbf f-\hat{\mathbf f}) \nu\cdot \nu\;\text{ on }\;T_r,
\]
where we used that $\dv^2\widehat{\mathbf f}= \dv^2 \mathbf f$, $\dv \bar{\vec g}=0$, and $(\mathbf{A} \nu\cdot \nu)u=\mathbf f \nu \cdot \nu$ on $T_2$.

Lemma \ref{lem2059mon} below shows that we have
\begin{equation}			\label{eq1947sat}
[v]_{C^{1,1}(B_{r/2}^+)} \le \frac{C}{r^2} \left(\fint_{B_{r}^+} \abs{v}^{\frac12}\right)^{2}.
\end{equation}

\begin{lemma}		\label{lem2059mon}
Let $\mathbf A_0$ be a constant symmetric matrix satisfying the ellipticity condition \eqref{ellipticity-nd} and $\mathbf f_0$ be a symmetric matrix whose entries are affine functions.
If $u\in L^p(B_2^+)$, for some $p>1$, satisfies
\[
\dv^2(\mathbf A_0 u)=0 \;\text { in }\; B_{2}^+,
\quad  (\mathbf A_0 \nu \cdot \nu)u=\mathbf f_0 \nu\cdot \nu \;\text{ on }\;T_2,
\]
then, we have
\[
\norm{D^2 u}_{L^\infty(B_1^+)}\le C \norm{u}_{L^{1/2}(B_2^+)},
\]
where $C=C(d, \lambda,\Lambda)$.
\end{lemma}

\begin{proof}
See Appendix.
\end{proof}

Note that the Taylor's theorem yields
\[
\sup_{x\in B_\rho^+}\, \abs{v(x)-v(0)-Dv(0)\cdot x}  \le C [v]_{C^{1,1}(B_\rho^+)} \rho^2,\quad \forall \rho \in (0, \tfrac12 r].
\]
Hence, by \eqref{eq1947sat}, we have
\[
\left(\fint_{B_{\rho}^+} \abs{v(x)-v(0)-Dv(0)\cdot x}^{\frac12}dx\right)^{2} \le C [v]_{C^{1,1}(B_{r/2}^+)} \rho^2 \le C \left(\frac{\rho}{r}\right)^2 \left(\fint_{B_r^+} \abs{v}^{\frac12}\right)^{2}.
\]
Note that the previous inequality remains unchanged if we replace $v$ by $v-\ell$ for any affine function $\ell(x)=\vec a \cdot x +b$.
Therefore, for any $\kappa \in (0, \frac12)$ and an affine function $\ell$, we have
\[
\left(\fint_{B_{\kappa r}^+} \abs{v(x)-v(0)-Dv(0)\cdot x}^{\frac12}dx\right)^{2}  \le C \kappa^2 \left(\fint_{B_r^+} \abs{v-\ell}^{\frac12}\right)^{2}.
\]

Next, define
\[
\varphi^+(r):=\frac{1}{r} \,\inf_{\substack{\vec a \in \bR^d\\b \in \bR}} \left(\fint_{B_{r}^+} \abs{u(x)-\vec a \cdot x - b}^{\frac12}dx\right)^{2},
\]
and, following \eqref{eq0224sat}, for any $\kappa \in (0,\frac12)$, we obtain
\[
\kappa r \varphi^+(\kappa r) \le C \kappa^2 \left(\fint_{B_r^+} \abs{u-\ell}^{\frac12}\right)^{2} + C \left(\kappa^2+\kappa^{-2d}\right) \left(\omega_{\rm coef}(2r) \norm{u}_{L^\infty(B_{2r}^+)}+ r\omega_{\rm dat}(2r)\right),
\]
where $C=C(d, \lambda, \Lambda)$.

Then, as in the previous sections, for any fixed number $\beta \in (0,1)$, we can choose a number $\kappa \in (0,\frac12)$ that depends solely on $d$, $\lambda$, $\Lambda$, and $\beta$, such that following inequality holds:
\[
\varphi^+(\kappa r) \le \kappa^{\beta} \varphi^+(r) + C \omega_{\rm coef}(2r)\, \frac{1}{2r}\, \norm{u}_{L^\infty(B_{2r}^+)} +C \omega_{\rm dat}(2r),
\]
where $C=C(d, \lambda, \Lambda, \beta)$.

Let $r_0\in (0, \frac12]$ be a number to be chosen later.
Similar to \eqref{eq1244sat}, define
\[
M^+_j(r_0):=\max_{0\le i <j}\, \frac{1}{k^i r_0} \,\norm{u}_{L^\infty(B_{\kappa^i r_0}^+)} \quad \text{for }j=1,2,\ldots.
\]
Then, similarly to \eqref{eq1114sat}, we obtain
\[
\varphi^+(\kappa^j r_0) \le \kappa^{\beta j} \varphi^+(r_0) + C M_j^+(2r_0)\, \tilde \omega_{\rm coef}(2\kappa^j r_0)  +C \tilde \omega_{\rm dat}(2\kappa^j r_0),\quad j=1,2,\ldots,
\]
where $\tilde \omega_{\rm coef}$ and $\tilde \omega_{\rm dat}$ are as defined in \eqref{eq1245sat}.

For $j=0,1,2,\ldots$, let $\vec a_j \in \bR^d$ and $b_j \in \bR$ be chosen so that
\[
\varphi^+(\kappa^j r_0)= \frac{1}{\kappa^j r_0}
\left(\fint_{B_{\kappa^j r_0}^+} \abs{u(x)-\vec a_j\cdot x - b_j}^{\frac12}dx\right)^{2}.
\]
Then, similarly to \eqref{eq0538tue} and \eqref{eq4.38}, we have
\begin{gather*}
\varphi^+(r_0) \le \frac{1}{r_0} \fint_{B_{r_0}^+} \abs{u},\\
\varphi^+(\kappa^j r_0) \le  \frac{\kappa^{\beta j}}{r_0} \fint_{B_{r_0}^+} \abs{u}+ C M_j^+(2r_0)\, \tilde \omega_{\rm coef}(2\kappa^j r_0)  +C \tilde \omega_{\rm dat}(2\kappa^j r_0),\quad j=1,2,\ldots.
\end{gather*}
Also, by the same argument as in the previous sections, we find
\begin{equation}			\label{eq1918tue}
\lim_{j\to \infty} b_j=0.
\end{equation}

Note that for any $\vec a \in \bR^d$, we have
\begin{equation}		\label{eq1810mon}
\fint_{B_r^+} \abs{\vec a \cdot x}^{\frac12} dx\ge  \abs{\vec a}^{\frac12} \inf_{\abs{\vec e}=1} \fint_{B_r^+}  \abs{\vec e \cdot x}^{\frac12}dx = \abs{\vec a}^{\frac12} r^{\frac12}  \inf_{\abs{\vec e}=1} \fint_{B_1^+} \abs{\vec e \cdot x}^{\frac12}dx= C\abs{\vec a}^{\frac12} r^{\frac12},
\end{equation}
where $C=C(d)>0$.
We use \eqref{eq1810mon}, instead of the identity \eqref{eq0229tue}, and repeat the same argument as in the previous sections.
We then obtain
\begin{multline}			\label{eq1748mon}
\abs{\vec a_k -\vec a_l} +\frac{\abs{b_k - b_l}}{\kappa^l r_0}  \le  \frac{C\kappa^{\beta l}}{r_0} \fint_{B_{r_0}^+} \abs{u}+ CM_k^+(2r_0) \int_0^{2\kappa^l r_0} \frac{\tilde \omega_{\rm coef}(t)}{t}\,dt\\
+ C \int_0^{2\kappa^l r_0} \frac{\tilde \omega_{\rm dat}(t)}{t}\,dt,\quad \text{for }\;k>l\ge 0,
\end{multline}
and for $j=1,2,\ldots$, we have
\begin{equation}			\label{eq1749mon}
\abs{\vec a_j} \le  \frac{C}{r_0} \fint_{B_{r_0}^+} \abs{u}+ CM_j^+(2r_0) \int_0^{2r_0} \frac{\tilde \omega_{\rm coef}(t)}{t}\,dt+C \int_0^{2r_0} \frac{\tilde \omega_{\rm dat}(t)}{t}\,dt.
\end{equation}

Moreover, we observe that $v(x)=u(x)-\vec a_j\cdot x-b_j$ satisfies
\begin{align*}
L^{*} v=\dv^2 (\mathbf{f} -\mathbf A (\vec a_j \cdot x + b_j))+\dv(\vec g+ (\vec a_j \cdot x + b_j) \vec b) + h-(\vec a_j \cdot x + b_j)c&\quad\text{in }\; B_2^{+}\\
(\mathbf A \nu \cdot \nu)v= (\mathbf{f} -\mathbf A (\vec a_j \cdot x + b_j)) \nu \cdot \nu &\quad\text{on }\; T_2.
\end{align*}

\begin{lemma}			\label{lem0921tue}
For $0<r \le \frac12$, we have
\[
\sup_{B_{r}^+}\, \abs{v} \le C \left\{ \left( \fint_{B_{2r}^+} \abs{v}^{\frac12} \right)^{2} + r\int_0^{2r} \frac{\omega_{\rm dat}(t)}{t}\,dt +(r\abs{\vec a_j}+\abs{b_j}) \int_0^{2r} \frac{\omega_{\rm coef}(t)}{t}\,dt \right\},
\]
where $C=C(d, \lambda, \Lambda, \omega_{\rm coef})$.
\end{lemma}
\begin{proof}
We replicate the proof of \cite[Lemma 2.26]{DEK18} in the same manner as Lemma~\ref{lem1146sat} was derived, focusing only on the main part of the argument.

As in the proof of Lemma \ref{lem1146sat}, our goal is to control the following quantity for $x_0 \in B_{3r/2}^+$ and $0<t\le r/4$:
\[
\phi(x_0, t):= \inf_{q \in \bR} \left(\fint_{B_t(x_0) \cap B_{2r}^+} \abs{v-q}^{\frac12} \right)^2.
\]

If $B_t(x_0)$ does not intersect $\partial\bR^d_+$, then $B_t(x_0) \subset B_{2r}^+$, and 
we follow the proof of Lemma \ref{lem1146sat}.
Otherwise, we replace $B_t(x_0) \cap B_{2r}^+$ with $B_{2t}^{+}(\bar x_0)$, where $\bar x_0$ is the projection of $x_0$ onto $\partial \bR^d_+$, and replicate the argument from \cite{DEK18}.
Thus, we only need to consider here the case where $x_0 \in \partial \bR^d_+$, so that $B_{t}(x_0)\cap B_{2r}^+=B_{t}^+(x_0)$.

We introduce the following notations:
\[
\mathbf F= \mathbf{f} -\mathbf A (\vec a_j \cdot x + b_j),\quad
\vec G=\vec g+ (\vec a_j \cdot x + b_j) \vec b,\quad
 H=h-(\vec a_j \cdot x + b_j)c.
\]

For $x_0 \in T_{3r/2}$ and $0<t \le r/4$, let $\bar{\mathbf A}$ and $\bar{\vec g}$ denote the averages of $\mathbf A$ and $\vec g$, respectively, over the half ball $B_{2t}^+(x_0)$.
Define $\mathbf L=(\ell^{kl})$ as follows:
\[
\ell^{kl}(x)=(D f^{kl})_{B_{2t}^+(x_0)}\cdot (x-x_0)+m^{kl} -(a^{kl})_{B_{2t}^+(x_0)} (\vec a_j \cdot x+b_j),
\]
where the constant $m^{kl}$ is chosen so that
\[
\int_{B_{2t}^+(x_0)} f^{kl}(y)- (D f^{kl})_{B_{2t}^+(x_0)}\cdot (y-x_0) -m^{kl}\,dy=0.
\]

We decompose $v$ as $v=v_1+v_2$, where $v_1 \in L^p(\mathcal{D}_t(x_0))$ (for some $p>1$) is the solution of the following problem:
\[
\left\{
\begin{aligned}
\dv^2(\bar{\mathbf A}v_1) &= \dv^2 \left(\mathbf{F}-\mathbf{L}-(\mathbf{A}- \bar{\mathbf A})v\right)+\dv(\vec G- \bar{\vec g}+\vec b v)+H-cv\;\mbox{ in }\;\mathcal{D}_t(x_0),\\
(\bar{\mathbf A}\nu\cdot \nu)v_1&=\left(\mathbf{F}-\mathbf{L}-(\mathbf{A}- \bar{\mathbf A}) v\right)\nu\cdot \nu \;\mbox{ on }\;\partial \mathcal{D}_t(x_0).
\end{aligned}
\right.
\]
Then, similar to \eqref{eq1649sat} (cf. \eqref{eq2025mon}), we obtain the following estimate:
\begin{multline}				\label{eq1458mon}
\left(\fint_{B_t^+(x_0)} \abs{v_1}^{\frac12} \right)^{2} \lesssim \fint_{B_{2t}^+(x_0)}\abs{\mathbf F - \mathbf L}+ \left(\fint_{B_{2t}^+(x_0)}\abs{\mathbf A-\bar{\mathbf A}}\right)\norm{v}_{L^\infty(B_{2t}^+(x_0))}
+ t \fint_{B_{2t}^+(x_0)} \abs{\vec G-\bar{\vec g}}\\
+\left(t \fint_{B_{2t}^+(x_0)}\abs{\vec b}\right)\norm{v}_{L^\infty(B_{2t}^+(x_0))} + t^2 \fint_{B_{2t}^+(x_0)} \abs{H} +  \left(t^2 \fint_{B_{2t}^+(x_0)}\abs{c}\right)\norm{v}_{L^\infty(B_{2t}^+(x_0))},
\end{multline}
where we used $B_{t}^+(x_0) \subset \mathcal{D}_t(x_0) \subset B_{2t}^+(x_0)$.

By definitions of $\mathbf F$, $\vec G$, and $H$, and applying the Poincar\'e inequality (cf.  \eqref{eq1025mon}), we obtain:
\begin{align*}
\fint_{B_{2t}^+(x_0)}\abs{\mathbf F-\mathbf L} &\le C t \fint_{B_{2t}^+(x_0)}\abs{D\mathbf f-(D\mathbf f)_{B_{2t}^+(x_0)}}+ (2r\abs{\vec a_j} +\abs{b_j})\fint_{B_{2t}^+(x_0)}\abs{\mathbf A-\bar{\mathbf A}},\\
t\fint_{B_{2t}^+(x_0)}\abs{\vec G-\bar{\vec g}} &\le t \fint_{B_{2t}^+(x_0)}\abs{\vec g-\bar{\vec g}}+ (2r\abs{\vec a_j} +\abs{b_j}) t\fint_{B_{2t}^+(x_0)}\abs{\vec b},\\	
t^2\fint_{B_{2t}^+(x_0)}\Abs{H} &\le t^2 \fint_{B_{2t}^+(x_0)}\abs{h}+ (2r\abs{\vec a_j} +\abs{b_j})t^2 \fint_{B_{2t}^+(x_0)}\abs{c}.
\end{align*}
Thus, from \eqref{eq1458mon}, we obtain the following (cf. \eqref{eq2026mon}):
\begin{equation}	\label{eq1200wed}
\left(\fint_{B^+_t(x_0)} \abs{v_1}^{\frac12} \right)^{2} \le C \omega_{\rm coef}(2t)\,\norm{v}_{L^\infty(B_t^{+}(x_0))} + Cr \omega_{\rm dat}(2t) + C(r\abs{\vec a_j}+\abs{b_j}) \omega_{\rm coef}(2t),
\end{equation}
where $C=C(d, \lambda,\Lambda)$.

On the other hand, since $\dv^2 \mathbf{L}=0$ and $\dv \bar{\vec g}=0$, it follows that $v_2=v-v_1$ satisfies
\[
\dv^2(\bar{\mathbf A}v_2) =0 \;\text{ in }\; B^+_t(x_0),\quad  (\bar{\mathbf A} \nu\cdot \nu)v_2=0\;\text{ on }\;T_t(x_0).
\]
Applying Lemma \ref{lem2059mon} with scaling, we obtain
\[
\norm{D^2 v_2}_{L^\infty(B_{t/2}^+(x_0))}\le C t^{-2-2d} \norm{v_2}_{L^{1/2}(B_{t}^+(x_0))}.
\]
Combining this with the interpolation inequality yields
\[
\norm{D v_2}_{L^\infty(B_{t/2}^+(x_0))}\le C t^{-1-2d}\norm{v_2}_{L^{1/2}(B_{t}^+(x_0))}.
\]
Notably, the above estimates remain valid if $v_2$ is replaced by $v_2 - \ell$ for any affine function $\ell$.
In particular, we have
\[
\norm{D v_2}_{L^\infty(B_{t/2}^+(x_0))}\le C t^{-1} \norm{v_2-c}_{L^{1/2}(B_{t}^+(x_0))},\quad \forall c \in \bR.
\]
The remainder of the proof involves a standard modification of the argument in \cite[Lemma 2.26]{DEK18}, which we leave to the reader.
\end{proof}

By using Lemma~\ref{lem0921tue} in place of Lemma~\ref{lem1146sat}, and repeating the same proof as in Section~\ref{sec3}, we observe that there exists $r_0=r_0(d, \lambda, \Lambda, \omega_{\rm coef}, \beta) \in (0,\frac12]$ such that
\begin{equation}			\label{eq2144mon}
\sup_{j \ge 1} M_j^+(2r_0) \le \frac{C}{r_0} \fint_{B_{4r_0}^+} \abs{u}+C \int_0^{2r_0} \frac{\tilde \omega_{\rm dat}(t)}{t}\,dt,
\end{equation}
where $C=C(d, \lambda, \Lambda, \omega_{\rm coef}, \beta)$.

Then, it follows from \eqref{eq1748mon} and \eqref{eq2144mon} that the sequence $\set{\vec a_j}$ converges to some $\hat{\vec a} \in\bR^d$, as it is a Cauchy sequence in $\bR^d$.
Moreover, recalling \eqref{eq1918tue}, we obtain
\begin{multline}			\label{eq2155mon}
\abs{\vec a_j -\hat{\vec a}} + \frac{\abs{b_j}}{\kappa^j r_0}  \le C \left( \kappa^{\beta j}+\int_0^{2\kappa^j r_0} \frac{\tilde \omega_{\rm coef}(t)}{t}\,dt \right) \frac{1}{r_0} \fint_{B_{4r_0}^+} \abs{u}\\
+C \left(\int_0^{2r_0} \frac{\tilde \omega_{\rm dat}(t)}{t}\,dt \right) \int_0^{2\kappa^j r_0} \frac{\tilde \omega_{\rm coef}(t)}{t}\,dt+ C \int_0^{2\kappa^j r_0} \frac{\tilde \omega_{\rm dat}(t)}{t}\,dt,
\end{multline}
where $C=C(d, \lambda, \Lambda,\omega_{\rm coef}, \beta)$.

Finally, similar to \eqref{eq2010sun}, we obtain from Lemma~\ref{lem0921tue} and \eqref{eq2155mon} that
\begin{equation}		\label{eq1652tue}
\frac{1}{r} \norm{u-\hat{\vec a}\cdot x}_{L^\infty(B_r^+)}  \le \varrho_{\rm coef}(r) \left\{\frac{1}{r_0} \fint_{B_{4r_0}^+} \abs{u} + \varrho_{\rm dat}(r_0)\right\}+\varrho_{\rm dat}(r),\quad  \forall r \in (0, r_0/4],
\end{equation}
where $\varrho_{\rm coef}(\cdot)$ and $\varrho_{\rm dat}(\cdot)$ are moduli of continuity defined similarly to \eqref{eq2011sun}.

In particular, \eqref{eq1652tue} implies that $u$ is differentiable at $0$, and from \eqref{eq1749mon} and \eqref{eq2144mon}, we obtain
\begin{equation}		\label{eq1107tue}
\abs{D u (0)} \le C \left(\frac{1}{r_0} \fint_{B_{4r_0}^+} \abs{u}+\int_0^{2r_0} \frac{\tilde \omega_{\rm dat}(t)}{t}\,dt  \right),
\end{equation}
where $C=C(d, \lambda, \Lambda, \omega_{\rm coef},\beta)$.\qed

\subsection{General case.}		\label{sec:c1alpha}
In this subsection, we consider the general case where $\partial\Omega$ is $C^{1,\alpha}$ for some $\alpha \in (0,1)$.  
To address this case, we introduce a special mixed Lebesgue space.
For $Q \subset \bR^d$, we say that $f \in L_{x'}^{p}L_{x_d}^{q}(Q)$ if
\[
\norm{f}_{L^{p}_{x'}L^{q}_{x_d}(Q)} = \left(\int_{\bR} \left(\int_{\bR^{d-1}} \abs{f \mathbbm{1}_Q(x',x_d)}^p\,dx'\right)^{q/p}\,dx_d\right)^{1/q}<\infty.
\]
Here, $x' \in \mathbb{R}^{d-1}$ and $x_d \in \mathbb{R}$.
When $p=\infty$ or $q=\infty$, the integrals are replaced by the essential supremum. 

Let $\Omega$ be a bounded $C^{1,\alpha}$ domain, and let $u$ be a bounded weak solution of
\[
L^* u=\dv^2 \mathbf{f}+\dv \vec g+ h\quad\text{in }\;\Omega,\qquad
u=\frac{\mathbf f \nu \cdot \nu}{\mathbf A \nu \cdot \nu}\quad\text{on }\;\partial\Omega.
\]

We flatten the boundary using the regularized distance function.
The regularized distance function $\phi$ on $\Omega$, introduced in \cite{Lie85}, satisfies the property that $\phi(x)$ is comparable to $\text{dist}\,(x,\partial\Omega)$ near the boundary, and $\phi \in C^{1,\alpha}(\overline{\Omega})\cap C^\infty(\Omega)$.

Without loss of generality, we assume $x^o=0\in \partial\Omega$ and that the $x_d$-direction is the normal direction at $0$.
We flatten the boundary $\partial\Omega$ near $0$ by introducing the change of variables $y=\vec\Phi(x)$, where
\[
\vec \Phi(x)=(\Phi^1(x),\ldots, \Phi^d(x))=(x_1,\ldots, x_{d-1}, \phi(x)).
\]
We denote
\[
\Omega_{r_0}:=\Omega \cap \Set{x: \max_{1\le i \le d}\,\abs{x_i}<r_0},\quad \widetilde\Omega_{r_0}:=\vec \Phi(\Omega_{r_0}) \subset \bR^d_+,
\]
and assume that $r_0$ is chosen so that $\vec \Phi:  \Omega_{r_0} \to \widetilde\Omega_{r_0}$ is a bijection.
We may also assume that $r_0 \ge 4$ by a similar reasoning as in Remark \ref{rmk_rescaling}.

Let $\vec \Psi$ denote the inverse of $\vec \Phi$, which has the form
\[
\vec \Psi(y)=(\Psi^1(y),\ldots, \Psi^d(y))=(y_1,\ldots, y_{d-1}, \psi(y)).
\]

Note that $\vec \Psi \in C^{1,\alpha}(\widetilde\Omega_{r_0})$.
Furthermore, observe that
\[
\det D\vec\Psi(y)= \frac{\partial \psi(y)}{\partial y_d}=D_d \psi(y)>0.
\]
Let $\tilde u(y):=u(\vec \Psi(y))$.
Then it satisfies
\[
\tilde L^*\tilde u= D_{kl}(\tilde a^{kl} \tilde u) - D_k(\tilde b^k \tilde u) +\tilde c \tilde u = D_{ij} \tilde f^{kl} + D_k g^k + \tilde h \;\;\text{ in } \; \widetilde \Omega_{r_0},
\]
where
\begin{equation}			\label{eq1237thu}
\begin{aligned}
\tilde a^{kl}(y)&=a^{ij}(\vec\Psi(y)) D_i \Phi^k(\vec\Psi(y))D_j \Phi^l(\vec\Psi(y)) D_d \psi(y),\\
\tilde b^k(y)&= \tilde b^k_1(y)+ \tilde b^k_2(y)\\
&:=b^{i}(\vec\Psi(y)) D_i \Phi^k(\vec\Psi(y)) D_d  \psi(y)+a^{ij}(\vec\Psi(y)) D_{ij} \Phi^k(\vec\Psi(y)) D_d  \psi(y),\\
\tilde c(y)&= c(\vec\Psi(y))  D_d  \psi(y),\\
\tilde f^{kl}(y) &= f^{ij}(\vec\Psi(y))D_i \Phi^k(\vec\Psi(y)) D_j \Phi^l(\vec\Psi(y)) D_d  \psi(y),\\
\tilde g^k(y)&=\tilde g^k_1(y)+\tilde g^k_2(y)\\
&:= -g^{i}(\vec\Psi(y)) D_i \Phi^k(\vec\Psi(y)) D_d \psi(y)-f^{ij}(\vec\Psi(y)) D_{ij}\Phi^k(\vec\Psi(y)) D_d  \psi(y),\\
\tilde h(y) &= h(\vec\Psi(y)) D_d  \psi(y).
\end{aligned}
\end{equation}
Notice that $D_{ij} \Phi^k=\partial^2 x_k/\partial x_i \partial x_j=0$ for $k=1,\ldots, d-1$, and  $D_{ij} \Phi^d=D_{ij} \phi$.
Therefore, by \cite[Lemma 2.4]{Safonov}, we have
\begin{equation}			\label{eq1317fri}
\abs{D^2 \vec \Phi (\vec \Psi(y))} = \abs{D^2 \phi (\vec \Psi(y))}  \le C y^{\alpha-1}_d.
\end{equation}

The boundary condition is transformed as
\[
\tilde a^{dd}u= \tilde f^{dd}\;\;\text{ on } \; \partial \widetilde \Omega_{r_0} \cap \partial \bR^d_+.
\]

Recall that $\vec \Phi$ and $\vec \Psi$ are $C^{1,\alpha}$ mappings.
Therefore, it follows from \eqref{eq1237thu} and \eqref{eq1317fri} that $\tilde{\mathbf A}=(\tilde a^{kl})\in \mathrm{DMO}(\widetilde \Omega_{r_0})$, $\tilde{\vec b}_1=(\tilde b_1^1,\ldots, \tilde b_1^d) \in L^{p_0}(\widetilde \Omega_{r_0})$, $\tilde{\vec b}_2=(\tilde b^1_2,\ldots, \tilde b^d_2) \in L_{x'}^{\infty}L_{x_d}^{q_0}(\widetilde \Omega_{r_0})$ with $q_0 \in (1,\frac{1}{1-\alpha})$, and $\tilde c \in L^{p_0/2}(\widetilde \Omega_{r_0})$.

Then, as in Remark \ref{rmk_omega}, H\"older's inequality implies that the following mapping satisfies the Dini condition:
\begin{equation}			\label{eq1648fri}
r\mapsto \omega_{\tilde{\mathbf A}}(r)+ r\sup_{x \in \Omega} \fint_{B_r(y^o) \cap \widetilde \Omega_{r_0}}\abs{\tilde{\vec b}} + r^2 \sup_{x \in \Omega} \fint_{B_r(y^o) \cap \widetilde\Omega}\abs{\tilde c}.
\end{equation}

Observe from \eqref{eq1237thu} that $\tilde {\mathbf f} \in C^{\alpha}(\widetilde \Omega_{r_0})$.
Also, it follows from \cite[Lemma 2.1]{DEK18} that $\tilde{\vec g}_1=(\tilde g^1_1,\ldots, \tilde g^d_1) \in \mathrm{DMO}(\widetilde \Omega_{r_0})$, and thus is bounded.
Moreover, similar to the case of $\tilde{\vec b}_2$, we have $\tilde{\vec  g}_2 \in L_{x'}^{\infty}L_{x_d}^{q_0}(\widetilde \Omega_{r_0})$ with $q_0 \in (1,\frac{1}{1-\alpha})$.
It is clear that $\tilde h \in L^{p_0}(\widetilde \Omega_{r_0})$.

Therefore, we conclude that the following mapping satisfies the Dini condition:
\begin{equation}			\label{eq1649fri}
r\mapsto \omega_{\tilde{\mathbf f}}(r)+ r\sup_{x \in \Omega} \fint_{B_r(y^o) \cap \widetilde \Omega_{r_0}}\abs{\tilde{\vec g}} + r^2 \sup_{x \in \Omega} \fint_{B_r(y^o) \cap \widetilde\Omega}\abs{\tilde h}.
\end{equation}

Observe that $\tilde{\vec b}$, $\tilde c$, $\tilde{\vec g}$, and $\tilde h$ all belong to $L^p(\widetilde\Omega)$ for some $p>1$, and that $\tilde u \in L^\infty(\widetilde\Omega)$.
Under these conditions, and noting that the mappings in \eqref{eq1648fri} and \eqref{eq1649fri} satisfy the Dini condition, the proof of \cite[Theorem 1.8]{DEK18} remains valid.
Consequently, $\tilde u$ is continuous in the closure of $\widetilde \Omega_{r_0/2}$.

On the other hand, we can still retain some control over the mean oscillations of $D\tilde{\mathbf f}$ and $\tilde{\vec g}$ by utilizing the fact that $f^{ij}(0)=0$, as we demonstrate below.

We use the following notation for the mean oscillations of $f$ in $B_t(y^o) \cap \widetilde \Omega_{r_0}$:
\[
\omega_{f}(t; y^o):= \fint_{B_t(y^o)\cap \widetilde \Omega_{r_0}} \,\Abs{f-(f)_{B_t(y^o) \cap \widetilde \Omega_{r_0}}}.
\]
Let us write:
\begin{align}					\nonumber
D \tilde f^{kl}(y) &= D_m f^{ij}(\vec\Psi(y)) D \Psi^m(y) D_i \Phi^k(\vec\Psi(y)) D_j \Phi^l(\vec\Psi(y)) D_d  \psi(y)\\		\nonumber
&\qquad+ f^{ij}(\vec\Psi(y))D_{im} \Phi^k(\vec\Psi(y)) D \Psi^m(y) D_j \Phi^l(\vec\Psi(y)) D_d  \psi(y)\\						\nonumber
&\qquad+f^{ij}(\vec\Psi(y))D_i \Phi^k(\vec\Psi(y)) D_{j m} \Phi^l(\vec\Psi(y)) D \Psi^m(y) D_d  \psi(y)\\						\nonumber
&\qquad+ f^{ij}(\vec\Psi(y))D_i \Phi^k(\vec\Psi(y)) D_j \Phi^l(\vec\Psi(y)) DD_{d}  \psi(y)\\
			\label{eq1712thu}
&=:\vec F_1^{kl}+\vec F_2^{kl}+\vec F_3^{kl}+\vec F_4^{kl}.
\end{align}
From \eqref{eq1712thu}, we observe that
\[
\omega_{D \tilde f^{kl}}(t; y^o)=\fint_{B_t(y^o)\cap \widetilde \Omega_{r_0}} \Abs{D \tilde f^{kl}-(D \tilde f^{kl})_{B_t(y^o)\cap \widetilde \Omega_{r_0}}} \le  \sum_{i=1}^4 \fint_{B_t(y^o)\cap \widetilde \Omega_{r_0}} \Abs{\vec F_i^{kl}-(\vec F_i^{kl})_{B_t(y^o)\cap \widetilde \Omega_{r_0}}}.
\]
Since $D \mathbf{f} \in \mathrm{DMO}$ and both $\vec \Phi$ and $\vec \Psi$ are $C^{1,\alpha}$ functions, by \cite[Lemma 2.1]{DEK18}, we have
\begin{equation}			\label{eq1142fri}
\fint_{B_t(y^o)\cap \widetilde \Omega_{r_0}} \Abs{\vec F_1^{kl}-(\vec F_1^{kl})_{B_t(y^o)\cap \Omega_{r_0}}}  \le C \omega_{D\mathbf f}(Ct) +Ct^\alpha.
\end{equation}
To handle the remaining terms $\vec F_2^{kl}$, $\vec F_3^{kl}$, and $\vec F_4^{kl}$, we utilize the assumption that  $f^{ij}(0)=0$, as follows.
For $y \in B_t(y^o) \cap \widetilde\Omega_{r_0}$, observe that
\[
\abs{f^{ij}(\vec \Psi(y))}=\abs{f^{ij}(\vec \Psi(y))-f^{ij}(\vec \Psi(0))} \le \norm{D(f^{ij}\circ \vec \Psi)}_\infty  \abs{y} \le C (t+\abs{y^o}).
\]
This inequality provides an upper bound for the terms involving $f^{ij}(\vec \Psi(y))$.
We recall that both $\vec \Phi$ and $\vec \Psi$ are $C^{1,\alpha}$ functions, and the estimate \eqref{eq1317fri} applies.
Additionally, since $\vec \Psi=\vec \Phi^{-1}$, it follows from \eqref{eq1317fri} that
\[
\abs{D^2\psi(y)} \le C y_d^{\alpha-1}.
\]
Using these observations, we now derive the bound
\begin{align*}
\sum_{i=2}^4 \fint_{B_t(y^o)\cap \widetilde\Omega_{r_0}} \Abs{\vec F_i^{kl}-(\vec F_i^{kl})_{B_t(y^o)\cap \widetilde\Omega_{r_0}}} & \le 2 \sum_{i=2}^4 \fint_{B_t(y^o)\cap \widetilde\Omega_{r_0}} \abs{\vec F_i^{kl}}\\
&\le C  (t+\abs{y^o}) t^{-d} \int_{B_t(y^o)\cap \widetilde\Omega_{r_0}} y_d^{\alpha-1}\,dy.
\end{align*}

In the case when $\dist(y^o, \partial \bR_+^d) >2t$, we have $y_d^{\alpha-1} \le t^{\alpha-1}$ in $B_t(y^o)\cap \widetilde\Omega_{r_0}$, and thus we obtain:
\[
\int_{B_t(y^o)\cap \widetilde\Omega_{r_0}} y_d^{\alpha-1}\,dy \le t^{\alpha-1} \abs{B_t}  \le Ct^{\alpha-1+d}.
\]
On the other hand, when $\dist(y^o, \partial \bR_+^d) \le 2t$, we have
\[
\int_{B_t(y^o)\cap \widetilde\Omega_{r_0}} y_d^{\alpha-1}\,dy  \le Ct^{d-1} \int_0^{3t} y_d^{\alpha-1}\,dy_d \le C t^{d-1+\alpha}.
\]
In both cases, we obtain
\[
\sum_{i=2}^4 \fint_{B_t(y^o)\cap \widetilde\Omega_{r_0}} \Abs{\vec F_i^{kl}-(\vec F_i^{kl})_{B_t(y^o)\cap \widetilde\Omega_{r_0}}} \le C (t+\abs{y^o}) t^{\alpha-1}.
\]

Therefore, together with \eqref{eq1142fri}, we obtain:
\begin{equation}			\label{eq1228fri}
\omega_{D \tilde{\mathbf f}}(t; y^o) \le C \omega_{D \mathbf{f}}(Ct)+ Ct^\alpha+ C(t+\abs{y^o})  t^{\alpha-1}.
\end{equation}
By similar reasoning, we derive the following as well:
\begin{equation}			\label{eq1227fri}
\omega_{\tilde{\vec g}}(t; y^o) \le 
C \omega_{\vec g}(Ct)+Ct^\alpha+ C (t+\abs{y^o}) t^{\alpha-1}.
\end{equation}

We now adapt the proof from Section \ref{sec:flat} to the current setting as follows.
We decompose $\tilde u$ as $\tilde u=v+w$, where $w$ is the $L^p$ weak solution, for some $p>1$, of the problem \eqref{eq1009sat} with the coefficients and data replaced by their counterparts denoted with tildes.
By setting $t=2r$ and $y^o=0$ in \eqref{eq1228fri} and \eqref{eq1227fri}, we obtain: 
\[
\fint_{B_{2r}^+} \abs{D\tilde{\mathbf f}-(D\tilde{\mathbf f})_{B_{2r}^+}} \le C \omega_{D \mathbf{f}}(Cr)+Cr^{\alpha},\quad
\fint_{B_{2r}^+} \abs{\tilde{\vec g}-(\tilde{\vec g})_{B_{2r}^+}} \le C \omega_{\vec g}(Cr)+Cr^{\alpha}.
\]

Thus, we observe that \eqref{eq2017fri} continues to hold when $u$ is replaced by $\tilde u$, and $\omega_{\rm coef}$ and $\omega_{\rm dat}$ are replaced by alternative Dini functions, denoted as $\tilde \omega_{\rm coef}$ and $\tilde \omega_{\rm dat}$.

In the proof of Lemma \ref{lem0921tue}, for $x_0 \in B_{3r/2}^+$ and $0<t\le r/4$, we consider two cases: when $B_t(x_0) \subset B_{2r}^+$ and when $x_0 \in T_{3r/2}$.

\vskip 5pt
\noindent
\textbf{Case 1:} $B_t(x_0) \subset B_{2r}^+$
\vskip 5pt

Observe that by setting $y^o=x_0$ in \eqref{eq1228fri} and \eqref{eq1227fri}, and noting that $t \le r/4$, we obtain:
\[
\fint_{B_t(x_0)} \abs{D\tilde{\mathbf f}-(D\tilde{\mathbf f})_{B_t(x_0)}} \le C \omega_{D \mathbf{f}}(Ct)+ Ct^\alpha + Crt^{\alpha-1},
\]
\[
\fint_{B_t(x_0)} \abs{\tilde{\vec g}-(\tilde{\vec g})_{B_t(x_0)}} \le C \omega_{\vec g}(Ct)+Ct^\alpha+Crt^{\alpha-1}.
\]
Thus, we have (recalling that $t\le r/4$):
\begin{align*}
t \fint_{B_t(x_0)} \abs{D\tilde{\mathbf f}-(D\tilde{\mathbf f})_{B_t(x_0)}} + t \fint_{B_t(x_0)} \abs{\tilde{\vec g}-(\tilde{\vec g})_{B_t(x_0)}}  & \le Ct \omega_{D \mathbf{f}}(Ct)+Ct \omega_{\vec g}(Ct)+Crt^{\alpha}\\
&\le Cr(\omega_{D \mathbf{f}}(Ct) + \omega_{\vec g}(Ct)+ t^\alpha).
\end{align*}

\vskip 5pt
\noindent
\textbf{Case 2:} $x_0 \in T_{3r/2}$
\vskip 5pt

In this case, by setting $y^o=x_0$ and replacing $t$ with $2t$ in \eqref{eq1228fri} and \eqref{eq1227fri}, and noting that $t \le r/4$, we obtain:
\[
\fint_{B_{2t}^+(x_0)} \abs{D\tilde{\mathbf f}-(D\tilde{\mathbf f})_{B_{2t}^+(x_0)}} \le C \omega_{D \mathbf{f}}(Ct)+ Ct^\alpha + Crt^{\alpha-1},
\]
\[
\fint_{B_{2t}^+(x_0)} \abs{\tilde{\vec g}-(\tilde{\vec g})_{B_{2t}^+(x_0)}} \le C \omega_{\vec g}(Ct)+Ct^\alpha +Crt^{\alpha-1}.
\]
Thus, we have (recalling that $t\le r/4$):
\begin{align*}
t \fint_{B_{2t}^+(x_0)} \abs{D\tilde{\mathbf f}-(D\tilde{\mathbf f})_{B_{2t}^+(x_0)}} + t \fint_{B_{2t}^+(x_0)} \abs{\tilde{\vec g}-(\tilde{\vec g})_{B_{2t}^+(x_0)}}  & \le Ct \omega_{D \mathbf{f}}(Ct)+Ct \omega_{\vec g}(Ct)+Crt^{\alpha}\\
&\le Cr(\omega_{D \mathbf{f}}(Ct) + \omega_{\vec g}(Ct)+ t^\alpha).
\end{align*}

Therefore, we observe that \eqref{eq2026mon} and \eqref{eq1200wed} remain valid when $\omega_{\rm coef}(\cdot)$ and $\omega_{\rm dat}(\cdot)$ are replaced by other functions satisfying the Dini condition -- namely, $\widehat \omega_{\rm coef}(\cdot)$ and $\widehat \omega_{\rm dat}(\cdot)$.
We note that $\widehat \omega_{\rm coef}$ and $\widehat \omega_{\rm dat}$ are determined by $\omega_{\rm coef}$ and $\omega_{\rm dat}$, respectively, as well as by the $C^{1,\alpha}$ properties of $\Omega$.

The remainder of the proof in Section \ref{sec:flat} remains unchanged.
In particular, we obtain estimates analogous to  \eqref{eq1652tue} and \eqref{eq1107tue} for $\tilde u=u\circ \vec \Psi$.
The theorem then follows from the observation that $u=\tilde u \circ \vec\Phi$, where $\vec \Phi$ is a $C^{1,\alpha}$ mapping. 
\qed

\begin{remark}
It is worth noting that the argument in Section \ref{sec:c1alpha} shows that \cite[Theorem 1.8]{DEK18} remains valid under the weaker assumption that $\Omega$ is a $C^{1,\alpha}$ domain for some $\alpha \in (0,1)$.
\end{remark}

\section{Applications to Green's function estimates}	\label{sec5}
In this section, we consider an elliptic operator $L_0$ in non-divergence form:
\begin{equation}			\label{eq1747sat}
L_0 u = a^{ij}D_{ij} u=\tr(\mathbf A D^2u),
\end{equation}
where $\mathbf A=(a^{ij})$ satisfies the condition \eqref{ellipticity-nd}.
We further assume that $\mathbf A$ has Dini mean oscillation and that $\Omega$ is a bounded $C^{1,\alpha}$ domain, for some $0<\alpha<1$, in $\bR^d$ with $d\ge 3$.
Here, we do not discuss Green's functions in two dimensions due to their logarithmic singularities, which require a distinct approach as outlined in \cite{DK21}.

Let $G(x,y)$ and $G^*(x,y)$ denote the Green's functions of the operators $L_0$ and $L_0^*$ in $\Omega$, respectively.
See \cite{HK20, KL21} for the construction of Green's functions.

We recall that for any $x,y\in \Omega$ with $x\neq y$, we have
\begin{equation}			\label{eq2005tue}
G(x,y)=G^*(y,x)
\end{equation}
and the following pointwise estimate holds:
\begin{equation}		\label{230323_eq1a}
\abs{G(x,y)} \le C_0\abs{x-y}^{2-d},
\end{equation}
where $C_0=C_0(d, \lambda, \Lambda, \omega_{\mathbf A},\Omega)>0$.
More precisely, the constant $C_0$ depends on $d$, $\lambda$, $\Lambda$, $\omega_{\mathbf A}$, the $C^{1,\alpha}$ character of $\partial\Omega$, and $\diam(\Omega)$.

As an application of Theorem~\ref{thm-main02}, we establish refined pointwise estimates for Green's function.

\begin{theorem}		\label{230321_thm1}
Let $\Omega$ be a bounded $C^{1,\alpha}$ domain, for some $0<\alpha<1$, in $\bR^d$ with $d \ge 3$.
Let $L_0$ be as defined in \eqref{eq1747sat}.
Assume that $\mathbf A$ satisfies the condition \eqref{ellipticity-nd} and has Dini mean oscillation.
Then, the Green's function $G(x,y)$ of $L_0$ in $\Omega$ satisfies the following pointwise estimates:
\begin{equation}		\label{230322_eq1}
\abs{G(x,y)} \le \frac{C}{\abs{x-y}^{d-2}}\left(1 \wedge \frac{d_x}{\abs{x-y}}\right)\left(1 \wedge \frac{d_y}{\abs{x-y}}\right), \quad  x\neq y,
\end{equation}
where $C=(d, \lambda, \Lambda, \omega_{\mathbf A},\Omega)>0$.
\end{theorem}

\begin{proof}
We shall first prove that
\begin{equation}		\label{230323_eq2}
\abs{G(x,y)} \le \frac{C}{\abs{x-y}^{d-2}} \left(1 \wedge \frac{d_y}{\abs{x-y}}\right),\quad x\neq y.
\end{equation}
It is sufficient to consider the case when $d_y \le \frac14 \abs{x-y}$, since otherwise, \eqref{230323_eq1a} immediately gives \eqref{230323_eq2} with $C=4C_0$.
Let us denote
\[
R:=\abs{x-y}
\]
and choose $y_0 \in \partial\Omega$ such that $\abs{y-y_0}=d_y$.
Then, by the triangle inequality, for $z \in \Omega \cap B_{R/2}(y_0)$, we have
\[
\abs{x-z} \ge \abs{x-y}-\abs{z-y_0}-\abs{y-y_0} \ge \tfrac14 \abs{x-y}.
\]
Therefore, by using \eqref{230323_eq1a}, we obtain
\begin{equation}			\label{eq16.57}
G(x,z) \le 4^{d-2}C_0 \abs{x-y}^{2-d},\quad \forall z \in \Omega \cap B_{R/2}(y_0).
\end{equation}
Additionally, since $u=G^*(\cdot, x)$ satisfies
\[
L^*_0 u=0 \;\text { in }\; \Omega\cap B_{R/2}(y_0),
\]
and $u(y_0)=G^*(y_0, x)=G(x, y_0)=0$, it follows from \eqref{eq2144mon} that we have
\begin{equation}		\label{eq1718tue}
\sup_{0<s \le r} \,\frac{1}{s} \norm{u}_{L^\infty(B_s(y_0)\cap \Omega)} \le \frac{C}{r} \fint_{B_{2r}(y_0)\cap \Omega} \abs{u},
\end{equation}
where $r=\frac14\abs{x-y} \wedge r_0$ and $r_0=r_0(d, \lambda, \Lambda, \omega_{\mathbf A}, \Omega)>0$; see Remark~\ref{rmk1203tue}.

We may assume that $d_y \le r_0$, since otherwise $r_0 < d_y \le \frac14 \abs{x-y}$, and again \eqref{230323_eq2} follows from \eqref{230323_eq1a} with $C=C_0 \diam(\Omega)/r_0$.
Then, since $d_y \le r \le \frac14 \abs{x-y}=R/4$, it follows from \eqref{eq1718tue}, \eqref{eq2005tue}, and \eqref{eq16.57} that
\[
\frac{1}{d_y}\abs{G(x,y)}  \le  \frac{C}{r} \fint_{\Omega \cap B_{2r}(y_0)} \abs{G(x,z)}\,dz \le \frac{4^{d-2}C_0}{r \abs{x-y}^{d-2}},
\]
which yields \eqref{230323_eq2} with $C=4^{d-2}C_0 \max(4, \diam(\Omega)/r_0)$.
We have thus established \eqref{230323_eq2}.

Next, we shall derive \eqref{230322_eq1} from \eqref{230323_eq2}.
As before, it is enough to consider the case when $d_x<\frac14 \abs{x-y}$.
Set $v=G(\cdot, y)$ and take $x_0 \in \partial\Omega$ such that $\abs{x-x_0}=d_x$.

Since $v$ satisfies
\[
L_0 v= a^{ij}D_{ij} v=0 \;\text { in }\; \Omega\cap B_{R/2}(x_0), \quad v=0 \;\text{ on }\; \partial \Omega\cap B_{R/2}(x_0),
\]
by the boundary gradient estimate for elliptic equations, we have
\begin{equation}		\label{230328_eq4}
\abs{v(x)}=\abs{v(x)-v(x_0)} \le  Cd_x R^{-1} \norm{v}_{L^\infty(\Omega\cap B_{R/2}(x_0))}.
\end{equation}
By the triangle inequality, for $z \in \Omega \cap B_{R/2}(x_0)$, we have $\tfrac{1}{4} \abs{x-y} \le \abs{z-y} \le \tfrac{7}{4} \abs{x-y}$.
This, together with \eqref{230328_eq4} and \eqref{230323_eq2}, gives
\[
\abs{G(x,y)} \le \frac{C d_x}{\abs{x-y}^{d-1}}\left(1 \wedge \frac{d_y}{\abs{x-y}}\right).
\]
The proof is complete.
\end{proof}

\begin{remark}
In addition to \eqref{230322_eq1}, we also have
\[
\abs{D_x G(x,y)} \le \frac{C}{\abs{x-y}^{d-1}} \left(1 \wedge \frac{d_y}{\abs{x-y}}\right).
\]
Moreover, if the boundary $\partial \Omega$ is of $C^{2, \rm{Dini}}$, then we have
\[
\abs{D_x^2 G(x,y)} \le \frac{C}{\abs{x-y}^{d}}\left(1 \wedge \frac{d_y}{\abs{x-y}}\right).
\]
These are straightforward consequences of \eqref{230322_eq1} and the standard elliptic estimates for non-divergence form equations, since $L G(\cdot, y)=0$ away from $y$. See \cite{HK20} for details.
\end{remark}

\section{Appendix}
\subsection{Proof of Lemma~\ref{lem01}}
We decompose $u$ as $u=u_1+u_2+u_3$, where $u_1 \in L^p(B)$ is the weak solution of
\[
\dv^2(\mathbf A_0 u_1)= \dv^2 \mathbf f\;\mbox{ in }\; B,\quad
(\mathbf A_0 \nu \cdot \nu) u_1 = \mathbf f \nu\cdot \nu \;\mbox{ on } \; \partial B,
\]
and $u_2 \in L^p(B)$ is the weak solution of
\[
\dv^2(\mathbf A_0 u_2)= \dv \vec g\;\mbox{ in }\; B,\quad
u_2 =0\;\mbox{ on } \; \partial B.
\]
Then, $u_3 \in L^p(B)$ is the weak solution of
\[
\dv^2(\mathbf A_0 u_3)= h\;\mbox{ in }\; B,\quad
u_3 =0\;\mbox{ on } \; \partial B.
\]

We note that the space $L^2$ in \cite[Lemma 2.1]{DK17} can be replaced by any $L^p$ with $p \in (1,\infty)$.
By adjusting the proof of \cite[Lemma 2.23]{DK17} accordingly, we obtain the following estimate:
\[
\Abs{\set{x \in B : \abs{u_1(x)} > t}}  \le \frac{C}{t} \int_{B} \abs{\mathbf f}.
\]

Moreover, a slight modification of the same proof will show that the corresponding estimates hold for $u_2$ and $u_3$.
Indeed, since $\dv^2(\mathbf A_0 u_2)= \dv(\mathbf A_0 Du_2)$, we have $u_2 \in W^{1,p}_0(B)$, and the map $T_2: \vec g \mapsto u_2$ is  a bounded linear operator on $L^p(B)$.
Similarly, since $\dv^2 (\mathbf A_0 u_3) = \tr(\mathbf A_0 D^2 u_3)$, it follows that $u_3 \in W^{2,p}(B)\cap W^{1,p}_0(B)$, and the map $T_3: h \mapsto u_3$ is a bounded linear operator on $L^p(B)$.

It suffices to show that both $T_2$ and $T_3$ satisfy the hypothesis of \cite[Lemma 2.1]{DK17} with $c=2$, where $L^2$ replaced by $L^p$, as noted earlier.

Let us first consider $T_2$.
For $\bar y \in B$ and $0<r<\frac12$, let $\vec b \in L^p(B)$ be supported in $B(\bar y, r) \cap B$ with mean zero.
Let $u \in W^{1,p}_0(B)$ be the solution to the problem
\[
\dv^2(\mathbf A_0 u) = \dv(\mathbf A_0 Du)=\dv  \vec b\;\mbox{ in }\; B;\quad
u=0\;\mbox{ on } \; \partial B.
\]

For any $R\ge 2r$ such that $B\setminus B(\bar y, R) \neq \emptyset$, and for $g \in C^\infty_c((B(\bar y,2R)\setminus B(\bar y,R))\cap B)$, let $v \in W^{2,p'}(B) \cap W^{1,p'}_0(B)$, where $1/p+1/p'=1$, be the solution to the problem
\[
\tr(\mathbf A_0 D^2 v)=  g\;\mbox{ in }\; B;\quad
v=0 \;\mbox{ on } \; \partial B.
\]
We note that the following the identity holds:
\[
\int_{B} u \, g = \int_{B}  \vec b \cdot Dv= \int_{B(\bar y ,r)\cap B} \vec b\cdot (Dv-(Dv)_{B(\bar y ,r)\cap B}).
\]
Since $g = 0$ in $B(\bar y, R) \cap B$, and $r\le R/2$, the standard interior and boundary regularity estimates for constant coefficient elliptic equations, along with the $L^p$ theory, yield
\begin{equation}			\label{eq0718sun}
\norm{D^2 v}_{L^\infty(B(\bar y, r)\cap B)} \lesssim
R^{-1-\frac{d}{p'}} \norm{Dv}_{L^{p'}(B(\bar y, R)\cap B)} \lesssim R^{-1-\frac{d}{p'}} \norm{D v}_{L^{p'}(B)} \lesssim R^{-1-\frac{d}{p'}} \norm{g}_{L^{p'}(B)}.
\end{equation}

Therefore, we have
\[
\Abs{\int_{(B(\bar y,2R)\setminus B(\bar y,R))\cap B} u\, g\,} \lesssim 
r R^{-1-\frac{d}{p'}}\,\norm{\vec b}_{L^1(B(\bar y,r) \cap B)}\, \norm{g}_{L^{p'}((B(\bar y,2R)\setminus B(\bar y,R))\cap B)}.
\]
Thus, by H\"older's inequality and the duality, we obtain
\[
\norm{u}_{L^1((B(\bar y, 2R)\setminus B(\bar y,R))\cap B)} \lesssim 
R^{\frac{d}{p'}} \norm{u}_{L^{p}((B(\bar y, 2R)\setminus B(\bar y,R))\cap B)}
\lesssim r R^{-1} \norm{\vec b}_{L^1(B(\bar y, r)\cap B)},
\]
which corresponds to \cite[(Eq.~(2.5)]{DK17} and implies that $T_2$ satisfies the hypothesis of \cite[Lemma 2.1]{DK17}.

Next, we turn to $T_3$.
For $\bar y \in B$ and $0<r<\frac12$, let $b \in L^p(B)$ be supported in $B(\bar y, r) \cap B$ with mean zero.
Let $u \in W^{2,p}(B)\cap W^{1,p}_0(B)$ be the solution of the problem
\[
\dv^2 (\mathbf A_0 u) =  \tr(\mathbf A_0 D^2 u)=b\;\mbox{ in }\; B;\quad
u=0\;\mbox{ on } \; \partial B.
\]
For any $R\ge 2r$ such that $B\setminus B(\bar y, R) \neq \emptyset$, and for $g \in C^\infty_c((B(\bar y,2R)\setminus B(\bar y,R))\cap B)$, let $v \in W^{2,p'}(B) \cap W^{1,p'}_0(B)$ be the solution to the problem
\[
\tr(\mathbf A_0 D^2 v)=  g\;\mbox{ in }\; B;\quad
v=0 \;\mbox{ on } \; \partial B.
\]
Note that we have the identity
\[
\int_{B} u \, g = \int_{B}  b v= \int_{B(\bar y ,r)\cap B} b\cdot (v-(v)_{B(\bar y ,r)\cap B}).
\]
Similar to \eqref{eq0718sun}, we derive
\[
\norm{D v}_{L^\infty(B(\bar y, r)\cap B)} \lesssim
R^{-1-\frac{d}{p'}} \norm{v}_{L^{p'}(B(\bar y, R)\cap B)} \lesssim R^{-1-\frac{d}{p'}} \norm{v}_{L^{p'}(B)} \lesssim R^{-1-\frac{d}{p'}} \norm{g}_{L^{p'}(B)}.
\]

Therefore, we have
\[
\Abs{\int_{(B(\bar y,2R)\setminus B(\bar y,R))\cap B} u\, g\,} \lesssim
r R^{-1-\frac{d}{p'}}\,\norm{b}_{L^1(B(\bar y,r) \cap B)}\, \norm{g}_{L^{p'}((B(\bar y,2R)\setminus B(\bar y,R))\cap B)}.
\]
The rest of the proof is similar, and we conclude that $T_3$ satisfies the hypothesis of \cite[Lemma 2.1]{DK17}.
\qed

\subsection{Proof of Lemma~\ref{lem2059mon}}
By the same reasoning as in the proof of \cite[Lemma 2.5]{DEK18}, the problem can be reduced to the following:
\begin{equation}		\label{230508@eq1}
\Delta u=0 \;\text { in }\; B_{2}^+,
\quad  u=\vec a\cdot x+b \;\text{ on }\;T_2,
\end{equation}
where $\vec a\in \bR^d$ and $b\in \bR$.
By differentiating \eqref{230508@eq1} twice in the tangential directions $x_k$ and $x_l$, where $k, l \in \set{1,\ldots, d-1}$, we obtain that $v_{kl}=D_{kl} u$ satisfies
\[
\Delta v_{kl}=0 \;\text { in }\; B_{2}^+,
\quad  v_{kl}=0 \;\text{ on }\;T_2.
\]
By the classical estimates for harmonic functions, we have
\[
\norm{v_{kl}}_{L^\infty(B_1^+)}+\norm{Dv_{kl}}_{L^\infty(B_1^+)}\le C\norm{v_{kl}}_{L^2(B_{3/2}^+)},
\]
which implies that
\begin{equation}		\label{230508@eq2}
\norm{D_{kl} u}_{L^\infty(B_1^+)}+\norm{DD_{kl} u}_{L^\infty(B_1^+)}\le C\norm{D_{kl} u}_{L^2(B_{3/2}^+)}, \quad k,l\in \set{1,\ldots, d-1}.
\end{equation}

Next, from the equation, we observe that
\[
D_{dd}u=-\sum_{k=1}^{d-1} D_{kk}u=0 \;\text{ on }\; T_2.
\]
Therefore, $v_d=D_d u$ satisfies
\[
\Delta v_{d}=0 \;\text { in }\; B_{2}^+,
\quad  D_d v_{d}=0 \;\text{ on }\;T_2.
\]
Again, by the classical estimates for harmonic functions, we have
\[
\norm{Dv_d}_{L^\infty(B_1^+)}+\norm{D^2 v_d}_{L^\infty(B_1^+)}\le C \norm{Dv_d}_{L^2(B_{3/2}^+)},
\]
which implies
\begin{equation}		\label{230508@eq3}
\norm{DD_d u}_{L^\infty(B_1^+)}+\norm{D^2 D_d u}_{L^\infty(B_1^+)}\le C \norm{DD_d u}_{L^2(B_{3/2}^+)}.
\end{equation}

Combining \eqref{230508@eq2} and \eqref{230508@eq3}, we obtain the following inequality:
\begin{equation}		\label{230508@eq4}
\norm{D^2 u}_{L^\infty(B_1^+)}+\norm{D^3 u}_{L^\infty(B_1^+)}\le C\norm{D^2 u}_{L^2(B_{3/2}^+)}.
\end{equation}
Thus, the desired estimate follows from \eqref{230508@eq4}, the interpolation inequality
\[
\norm{D^2 u}_{L^2(B_{3/2}^+)}+\norm{u}_{L^\infty(B_{3/2}^+)}\le \varepsilon \norm{D^3 u}_{L^\infty(B_{3/2}^+)}+C \norm{u}_{L^{1/2}(B_{3/2}^+)},
\]
and a standard iteration argument.
\qed


\end{document}